\documentclass[12pt]{amsproc}

\usepackage{fullpage}
\usepackage[utf8]{inputenc}
\usepackage[T1]{fontenc}
\usepackage{graphicx,subfigure}
\usepackage{amsmath,amsfonts,amssymb,mathtools}
\usepackage{stmaryrd}
\usepackage{mathrsfs,dsfont}
\usepackage{amsthm}
\usepackage{mathabx}
\usepackage{tabularx}
\usepackage{float}

\usepackage{hyperref}

\usepackage{enumerate}

\usepackage{color}

\newcommand{\E}{\mathbb{E}}
\newcommand{\R}{\mathbb{R}}
\newcommand{\N}{\mathbb{N}}

\newtheorem{theo}{Theorem}[section]

\newtheorem{cor}[theo]{Corollary}
\newtheorem{rem}[theo]{Remark}

\newtheorem{propo}[theo]{Proposition}
\newtheorem{lemma}[theo]{Lemma}

\newtheorem{ass}{Assumption}

\usepackage{algorithm}
\usepackage{algorithmic}

\begin{document}

\title{Approximation of the invariant distribution for a class of ergodic SPDEs using an explicit tamed exponential Euler scheme}
\author{Charles-Edouard Br\'ehier}
\address{Univ Lyon, Université Claude Bernard Lyon 1, CNRS UMR 5208, Institut Camille Jordan, 43 blvd. du 11 novembre 1918, F-69622 Villeurbanne cedex, France}
\email{brehier@math.univ-lyon1.fr}

\date{}


\begin{abstract}
We consider the long-time behavior of an explicit tamed exponential Euler scheme applied to a class of parabolic semilinear stochastic partial differential equations driven by additive noise, under a one-sided Lipschitz continuity condition. The setting encompasses nonlinearities with polynomial growth. First, we prove that moment bounds for the numerical scheme hold, with at most polynomial dependence with respect to the time horizon. Second, we apply this result to obtain error estimates, in the weak sense, in terms of the time-step size and of the time horizon, to quantify the error to approximate averages with respect to the invariant distribution of the continuous-time process. We justify the efficiency of using the explicit tamed exponential Euler scheme to approximate the invariant distribution, since the computational cost does not suffer from the at most polynomial growth of the moment bounds. To the best of our knowledge, this is the first result in the literature concerning the approximation of the invariant distribution for SPDEs with non-globally Lipschitz coefficients using an explicit tamed scheme.
\end{abstract}

\maketitle

\section{Introduction}

In the last $25$ years, the analysis of numerical methods for Stochastic Partial Differential Equations (SPDEs) has been a very active research field. Pionnering work have focused on the so-called strong convergence of numerical schemes for equations with Lipschitz continuous nonlinearities, and in the last decade many results concerning convergence of schemes for equations with non-globally Lipschitz continuous nonlinearities, and weak convergence, have been obtained. We refer to the monograph~\cite{LPS} for a pedagogical introduction to this field of research.

In this article, we consider some semilinear parabolic SPDEs of the type
\[
dX(t)=AX(t)dt+F(X(t))dt+dW^Q(t),
\]
as written in the framework of stochastic evolution equations, see~\cite{DPZ}. Precise assumptions concerning the linear operator $A$, the nonlinearity $F$ and the $Q$-Wiener process $W^Q$ are given below (Section~\ref{sec:setting}). Under appropriate assumption, this process admits a unique invariant probability distribution $\mu_\star$, such that
\[
\E[\varphi(X(T))]\underset{T\to\infty}\to \int \varphi d\mu_\star,
\]
exponentially fast, for any initial condition $X(0)$ and any real-valued Lipschitz continuous function $\varphi$ (see Section~\ref{sec:recall}). We study the question of approximating the invariant distribution $\mu_\star$ using a numerical scheme. The main novelty of this article is to show that an explicit scheme can be used, without loss of computational efficiency, even if the nonlinearity $F$ is not globally Lipschitz continuous.

Let us review the existing literature concerning the numerical approximation of invariant distributions for parabolic semilinear SPDEs -- see also the preprint~\cite{Boyaval-Martal-Reygner} where stochastic viscous scalar conservation laws are considered, and the monograph~\cite{Hong-Wang} and references therein where some stochastic Schr\"odinger equations are studied. In~\cite{C-E:14}, parabolic semilinear SPDEs, with Lipschitz nonlinearity, driven by space-time white noise, are considered; temporal discretization is performed using a linear implicit Euler scheme, and weak error estimates which are uniform in time are obtained using a Kolmogorov equation approach. In~\cite{C-E:17}, the same framework has been considered, for full-discretization schemes (using a finite element method for spatial discretization); error estimates are obtained using a Poisson equation approach. A full-discretization scheme based on an exponential Euler scheme has been considered in~\cite{Chen-Gan-Wang:20}. A postprocessed integrator has been proposed in~\cite{C-E:16_2} in order to increase the order of convergence. For non globally Lipschitz continuous nonlinearities, the only existing result is the preprint~\cite{Cui-Hong-Sun}, where the authors use a fully implicit scheme. Note that the literature is also limited concerning the analysis of the weak error on finite time intervals when applied to SPDEs with non-globally Lipschitz nonlinearity: see~\cite{C-E:20} where a splitting scheme is applied for the Allen-Cahn equation (cubic nonlinearity), and also~\cite{CuiHong:19} and~\cite{Can-Gan-Wang}.

In this article, we consider exponential integrators to deal with the linear part. As demonstrated in~\cite{Beccari+} (see also the monograph~\cite{HutzenthalerJentzen}), using a simple explicit exponential Euler scheme (like for instance in~\cite{JentzenKloeden}) of the type
\[
X_{n+1}=e^{\Delta tA}(X_n+\Delta tF(X_n)+\Delta W_n^Q),
\]
where the nonlinearity is discretized explictly, is not appropriate due to the loss of moment bounds, which would be essential in the proof of convergence. Many recipes have been proposed to overcome this issue (we refer to~\cite{HutzenthalerJentzen} for a general analysis in the case of finite-dimensional Stochastic Differential Equations). In this article, we consider the following explicit tamed exponential Euler scheme (see Equation~\eqref{eq:scheme})
\[
X_{n+1}=e^{\Delta tA}X_{n}+(-A)^{-1}(I-e^{\Delta tA})\frac{F(X_n)}{1+\Delta t\|F(X_n)\|_{L^2}}+e^{\Delta tA}\Delta W_n^Q,
\]
where $\Delta W_n^Q$ are Wiener increments. This type of scheme has already been studied in~\cite{GyongySabanisSiska} and~\cite{Wang} for instance. In those works, a finite time horizon $T\in(0,\infty)$ is fixed, and the authors look at strong convergence. To the best of our knowledge, neither weak convergence rates nor long-time behavior have been considered for this type of explicit tamed scheme for SPDEs in the literature so far. The objective of this article is to prove that the explicit tamed scheme can be employed to approximate the invariant distribution $\mu_\star$, with precise error estimates and analysis of the computational cost.

The contribution of this article is twofold. First, moment bounds of the type
\[
\underset{0\le n\Delta t\le T}\sup~\bigl(\E[\|X_n\|_{L^\infty}^m\bigr)^{\frac1m}\le C(\|x_0\|_{L^\infty})(1+T)
\]
are obtained, for all $T\in(0,\infty)$. See Theorem~\ref{theo:moment} for a precise statement. The moment bound is uniform with respect to the time-step size $\Delta t$. It is important to note that the upper bound is not uniform with respect to the final time $T$.

Second, weak error estimates of the type
\[
\big|\E[\varphi(X_n)]-\E[\varphi(X(t_n))]\big|\le C_\alpha(\varphi,\|x_0\|_{L^\infty})(1+T^Q)\Delta t^{2\alpha},
\]
for some $Q\in\N$, where $\alpha\in(0,\overline{\alpha})$ and $\overline{\alpha}$ is a parameter which quantifies the regularity of the process. If the equation is driven by a cylindrical Wiener process ($Q=I$, space-time white noise), $\overline{\alpha}=\frac14$, whereas $\overline{\alpha}=\frac12$ if the equation is driven by trace-class noise. See Theorem~\ref{theo:error} for a precise statement. The weak error estimate is not uniform with respect to time $T$.

Even if the upper bounds show some dependence with respect to time $T$, the polynomial growth is sufficiently slow (compared with an exponential growth) not to damage the performance of the scheme when one is interested in the approximation of the invariant distribution. Since the convergence to equilibrium is exponentially fast (with respect to $T$) in the models considered here, the analysis of the computational cost (see Corollary~\ref{cor:cost}) reveals that there is no loss in the efficiency, when compared with a situation where uniform moment bounds and error estimates hold (for instance, if the nonlinearity is globally Lipschitz as in~\cite{C-E:14}, or if an implicit integrator is employed as in~\cite{Cui-Hong-Sun}). Since the tamed exponential Euler scheme is explicit, in practice its application is simpler and the cost per step is lower, thus the results of this article justify its efficiency for the approximation of the invariant distribution for SPDEs.

To the best of our knowledge, the long-time behavior of moment bounds and weak error estimates for explicit tamed Euler schemes applied to SPDEs has not been studied in the literature so far, and the results above show that these schemes are effective to numerically approximate the invariant distribution of SPDEs. Note that only upper bounds are proven, and it may happen that in fact the polynomial growth is not optimal and that uniform in time upper bounds may be proved by other techniques. Some numerical experiments have been performed and have not been sufficient to exhibit the polynomial growth (thus they are not reported here). Studying whether the polynomial growth is optimal or whether uniform bounds can be obtained is left for future works. The combination of the results of this article with the analysis of spatial discretization, and the application of (multilevel) Monte-Carlo methods is also left as an open question.

This article is organized as follows. The setting is presented and the main assumptions are stated in Section~\ref{sec:setting}. Preliminary results concerning the long-time behavior of the solution of the SPDE are given in Section~\ref{sec:recall}. The explicit tamed exponential Euler scheme is defined in Section~\ref{sec:main}, where the two main results, Theorem~\ref{theo:moment} (moment bounds) and Theorem~\ref{theo:error} (error estimate) are stated and discussed. Section~\ref{sec:proofTh1} is devoted to the proof of Theorem~\ref{theo:moment}, whereas Section~\ref{sec:proofTh2} is devoted to the proof of Theorem~\ref{theo:error}.

\section{Setting}\label{sec:setting}

For every $p\in[1,\infty]$ let $L^p=L^p(0,1)$ be the Banach space of real-valued functions with finite $L^p$ norm, and let the associated norm be denoted by $\|\cdot\|_{L^\infty}$. When $p=2$, $H=L^2$ is an Hilbert space, with inner product denoted by $\langle\cdot,\cdot\rangle$. Let $\N$ denote the set of integers and $\N_0=\N\cup\{0\}$.

The time-step size of the numerical schemes is denoted by $\Delta t$. The a priori bounds and error estimates will be stated below for $\Delta t\in(0,\Delta t_0]$, where $\Delta t_0$ is an arbitrary parameter. Once the value of the time-step size $\Delta t$ is fixed, the following notation is used: for all $n\in\N_0$, $t_n=n\Delta t$, and for all $t\in[0,\infty)$, $\ell(t)=n$ if and only if $t\in[t_n,t_{n+1})$.

The values of constants may change from line to line in the proofs, however their are assumed not to depend on quantities such as the time-step size $\Delta t$, the initial condition $x_0$.

\subsection{Linear operator}

Set $e_n=\sqrt{2}\sin(n\pi\cdot)$ and $\lambda_n=(n\pi)^2$, for all $n\in\N$. Then $\bigl(e_n\bigr)_{n\in\N}$ is a complete orthonormal system of $H$.

The unbounded self-adjoint linear operator on $H$ defined by
\[
A=-\sum_{n\in\N}\langle \cdot,e_n\rangle e_n
\]
is the realization of the Laplace operator on $(0,1)$, with homogeneous Dirichlet boundary conditions. The associated semi-group $\bigl(e^{tA}\bigr)_{t\ge 0}$, such that $e^{tA}=\sum_{n\in\N}e^{-\lambda_n t}\langle \cdot,e_n\rangle e_n$, is such that $t\mapsto e^{tA}x$ is the solution of the heat equation with homogeneous Dirichlet boundary conditions. For every $p\in[2,\infty)$, $\bigl(e^{tA}\bigr)_{t\ge 0}$ also defines a semi-group on $L^p$. For every $\alpha\in[-1,1]$, let $(-A)^{\alpha}$ be the operator defined on $H$ by
\[
(-A)^\alpha=\sum_{n\in\N}\lambda_n^\alpha \langle \cdot,e_n\rangle e_n.
\]

Let us state some inequalities which are employed in the sequel. We refer for instance to~\cite{Cerrai} and to~\cite{Triebel}.
\begin{itemize}
\item For all $x\in L^2$, one has
\begin{equation}\label{gapL2}
\|e^{tA}x\|_{L^2}\le e^{-\lambda_1 t}\|x\|_{L^2}.
\end{equation}
\item For every $\alpha\in[0,1]$, there exists $C_\alpha\in(0,\infty)$ such that for all $x\in L^2$ and $t>0$, one has
\begin{equation}\label{eq:expoAalpha}
\|(-A)^\alpha e^{tA}x\|_{L^2}\le C_\alpha \min(t,1)^{-\alpha}\|x\|_{L^2}.
\end{equation}
\item For every $\alpha\in[0,1]$, there exists $C_\alpha\in(0,\infty)$ such that for all $x\in L^2$ and $t,s\ge 0$, one has
\begin{equation}\label{eq:tempA}
\|e^{tA}x-e^{sA}x\|_{L^2}\le C_\alpha |t-s|^\alpha \|(-A)^\alpha x\|_{L^2}.
\end{equation}
\item There exists $c,C\in(0,\infty)$ such that for all $x\in L^2$ and all $t>0$, one has
\begin{equation}\label{eq:ineqLinftyL2}
\|e^{tA}x\|_{L^\infty}\le C\min(t,1)^{-\frac14}e^{-ct}\|x\|_{L^2},
\end{equation}
and for all $x\in L^1$ and all $t>0$, one has
\begin{equation}\label{eq:ineqLinftyL1}
\|e^{tA}x\|_{L^\infty}\le C\min(t,1)^{-\frac12}e^{-ct}\|x\|_{L^1}.
\end{equation}
\item There exists $c\in(0,\infty)$ such that, for all $p\in(2,\infty)$, there exists $C_p\in(0,\infty)$, such that for all $x\in L^p$ and all $t>0$, one has
\begin{equation}\label{eq:expoLp}
\|e^{tA}x\|_{L^p}\le C_pe^{-ct}\|x\|_{L^p}.
\end{equation}
\item There exists $C\in(0,\infty)$ such that for all $x\in L^\infty$ and lal $t\ge 0$ one has
\begin{equation}\label{eq:LinftyLinfty}
\|e^{tA}x\|_{L^\infty}\le C\|x\|_{L^\infty}.
\end{equation}
\item For all $\kappa\in(\frac14,1)$, there exists $C_\kappa\in(0,\infty)$ such that for all $x\in L^2$ one has
\begin{equation}\label{eq:sob}
\|x\|_{L^\infty}\le C_\kappa \|(-A)^\kappa x\|_{L^2}.
\end{equation}
\item For all $p\in[2,\infty)$, there exists $\Lambda_p\in(0,\infty)$ such that for all $x\in L^p$, such that $Ax\in L^p$, one has
\begin{equation}\label{eq:gap}
\langle Ax,x|x|^{p-2}\rangle\le -\Lambda_p\|x\|_{L^p}^p.
\end{equation}
\item For all $\alpha\in(0,\frac12)$, $\epsilon>0$, with $\alpha+\epsilon<\frac12$, there exists $C_{\alpha,\epsilon}\in(0,\infty)$ such that
\begin{equation}\label{eq:ineqproduct1}
|(-A)^{-\alpha-\epsilon}(xy)|_{L^1}\le C_{\alpha,\epsilon}|(-A)^{\alpha+\epsilon}x|_{L^2}|(-A)^{-\alpha}y|_{L^2}.
\end{equation}
\item For all $\alpha\in(0,\frac12)$, $\epsilon>0$, with $\alpha+2\epsilon<\frac12$, there exists $C_{\alpha,\epsilon}\in(0,\infty)$ such that, if $\psi:\R\to\R$ is Lipschitz continuous, 
\begin{equation}\label{eq:ineqproduct2}
|(-A)^{\alpha+\epsilon}\psi(\cdot)|_{L^2}\le C_{\alpha,\epsilon}[\psi]_{\rm Lip}|(-A)^{\alpha+2\epsilon}\cdot|_{L^2},
\end{equation}
where $[\psi]_{\rm Lip}=\underset{z_1\neq z_2}\sup~\frac{|\psi(z_2)-\psi(z_1)|}{|z_2-z_1|}$.
\end{itemize}

Note that $\Lambda_2=\lambda_1$, and for all $p\ge 2$, one has the lower bound $\Lambda_p\ge \lambda_1\frac{4(p-1)}{p^2}$. Indeed using an integration by parts argument, one has
\begin{align*}
\langle Ax,x|x|^{p-2}\rangle&=-\langle \nabla x,\nabla (x|x|^{p-2})\rangle\\
&=-(p-1)\langle \nabla x,\nabla x |x|^{p-2}\rangle\\
&=-(p-1)\||^{\frac{p-2}{2}}\nabla x\|_{L^2}^2\\
&=-\frac{4(p-1)}{p^2}\|\nabla (|x|^{\frac{p}{2}})\|_{L^2}^2\\
&\le -\lambda_1\frac{4(p-1)}{p^2}\||x|^{\frac{p}{2}}\|_{L^2}^2\\
&=-\lambda_1\frac{4(p-1)}{p^2}\|x\|_{L^p}^p.
\end{align*}

\subsection{Nonlinearity}

Let us now give the assumptions concerning the nonlinearity: $F$ is the Nemytskii operator associated with a real-valued function $f:\R\to\R$ which is assumed to be of class $\mathcal{C}^2$, with at most polynomial growth in the following sense.
\begin{ass}\label{ass:poly}
There exists a real number $q\ge 2$ such that
\begin{equation}
\underset{z\in\R}\sup~\frac{|f(z)|+|f'(z)|+|f''(z)|}{1+|x|^q}<\infty.
\end{equation}
\end{ass}

To study the long time behavior, the following one-sided Lipschitz continuity condition is enforced.
\begin{ass}\label{ass:condition}
There exists $\gamma>0$ such that, if $p\in\{2,q\}$, one has
\begin{equation}\label{eq:condition}
\langle Ay,y|y|^{p-2}\rangle+\langle F(y+z)-F(z),y|y|^{p-2}\rangle\le -\gamma\|y\|_{L^p}^p,
\end{equation}
whenever $y,z\in L^p$ are such that the left-hand side of~\eqref{eq:condition} makes sense, where $q$ is given by Assumption~\ref{ass:poly}. 
\end{ass}
Let $\lambda_F=\underset{z\in\R}\sup~f'(z)$. When $\lambda_F<\infty$, the real-valued function $f$ satisfies a standard one-sided Lipschitz continuity condition. The condition~\eqref{eq:condition} for $p=2$ is satisfied if $\lambda_F<\lambda_1$. When $\lambda_F<0$, owing to~\eqref{eq:gap}, the condition~\eqref{eq:condition} is satisfied for all $p\in[2,\infty)$.

Assumptions~\ref{ass:poly} and~\ref{ass:condition} are satisfied  for instance if $f$ is a polynomial function, such that $\lambda_F=\underset{z\in\R}\sup~f'(z)$ is finite and sufficiently small.

\subsection{Wiener process and the stochastic convolution}

Let $\bigl(\beta_j\bigr)_{j\in\N}$ be a sequence of independent standard real-valued Wiener processes, defined on a probability space $\bigl(\Omega,\mathcal{F},\mathbb{P}\bigr)$ which satisfies the usual conditions. Let the associated expectation operator be denoted by $\E[\cdot]$.

The stochastic evolution equation is driven by a $Q$-Wiener process $\bigl(W^Q(t)\bigr)_{t\ge 0}$, with covariance operator $Q$. It is assumed that $Q=\sum_{j\in\N}q_j\langle \cdot,\tilde{e}_j\rangle \tilde{e}_j$, where $q_j\in[0,\infty)$ for all $j\in\N$, and $\bigl(\tilde{e}_j\bigr)_{j\in\N}$ is a complete orthonormal system of $H$.

We refer to the monographs~\cite{Cerrai} and~\cite{DPZ} for the general theory of stochastic evolution equations and properties of their solutions.

The stochastic convolution is the process defined by
\begin{equation}\label{eq:Z}
Z(t)=\int_0^t e^{(t-s)A}dW^Q(s)=\sum_{j\in\N}\sqrt{q_j}\int_{0}^{t}e^{(t-s)A}\tilde{e}_j d\beta_j(s),
\end{equation}
for all $t\ge 0$. The stochastic convolution is the mild solution of the linear equation driven by additive noise:
\[
dZ(t)=AZ(t)dt+dW^Q(t),\quad Z(0)=0.
\]
To justify well-posedness for the stochastic convolution and establish moment bounds for the solution in the $L^\infty$ norm, the following assumption is required.
\begin{ass}\label{ass:momentZ}
The stochastic convolution $\bigl(Z(t)\bigr)_{t\ge 0}$ is well-defined and takes values in $L^\infty$. Moreover, uniform in time moment bounds are satisfied in the $L^\infty$-norm: for all $m\in\N$, 
\begin{equation}
\underset{t\ge 0}\sup~\E[\|Z(t)\|_{L^\infty}^m]<\infty.
\end{equation}
\end{ass}
When the covariance operator $Q$ and the linear operator $A$ commute, {\it i.e.} when $\tilde{e}_j=e_j$ for all $j\in\N$, it is possible to sample the $H$-valued Gaussian random variable $Z(t)$ exactly in distribution. However, in general this is not possible and a numerical approximation scheme needs to be employed. In this article, exponential Euler integrators are used. Let $\Delta t>0$ denote the time step-size, and define the sequence of random variables $\bigl(Z_n^{\Delta t}\bigr)_{n\in \N_0}$ by  
\begin{equation}\label{eq:Zn}
Z_{n+1}^{\Delta t}=e^{\Delta tA}\bigl(Z_n^{\Delta t}+\Delta W_n^Q\bigr),
\end{equation}
with $Z_0^{\Delta t}=0$, and where $\Delta W_n^Q=W^Q(t_{n+1})-W^Q(t_n)$ are the Wiener increments. For every $n\in\N_0$ and $t\in[t_n,t_{n+1}]$, set
\begin{equation}\label{eq:Ztilde}
\tilde{Z}^{\Delta t}(t)=e^{(t-t_n)A}\bigl(Z_n^{\Delta t}+W^Q(t)-W^Q(t_n)\bigr).
\end{equation}
The following discrete-time versions of~Assumption~\ref{ass:momentZ} is introduced.
\begin{ass}\label{ass:momentZtilde}
For any time-step size $\Delta t\in(0,\Delta t_0]$, the process $\bigl(Z_n^{\Delta t}\bigr)_{n\in \N_0}$ takes values in $L^\infty$. Moreover, some moment bounds which are uniform with respect to time and to the time-step size are satisfied: for all $m\in\N$,
\begin{equation}
\underset{\Delta t\in(0,\Delta t_0]}\sup~\underset{t\ge 0}\sup~\E[\|\tilde{Z}^{\Delta t}(t)\|_{L^\infty}^m]<\infty.
\end{equation}
\end{ass}

It remains to define a parameter $\overline{\alpha}$ to express the order of convergence (in the weak sense) of the numerical scheme.
\begin{ass}\label{ass:alpha}
Assume that there exists $\alpha>0$ such that
\begin{equation}
\sum_{j=1}^{\infty}q_j\|(-A)^{2\alpha-1}\tilde{e_j}\|_{L^2}^2<\infty.
\end{equation}
\end{ass}

The parameter $\overline{\alpha}$ is then defined by
\begin{equation}\label{eq:alphabar}
\overline{\alpha}=\sup~\{\alpha\in(0,\frac12];~\sum_{j=1}^{\infty}q_j\|(-A)^{2\alpha-1}\tilde{e_j}\|_{L^2}^2<\infty\}.
\end{equation}
For all $\alpha\in(0,\overline{alpha})$, one then has
\begin{equation}\label{eq:regul_alpha}
\underset{t\ge 0}\sup~\E[\|(-A)^{\alpha}Z(t)\|_{L^2}^2]+\underset{t\ge 0}\sup~\E[\|(-A)^{\alpha}\tilde{Z}(t)\|_{L^2}^2]<\infty.
\end{equation}

Assumptions~\ref{ass:momentZ},~\ref{ass:momentZtilde} and~\ref{ass:alpha} are compatible and are satisfied for many examples. For instance, they are satisfied in the case of a cylindrical Wiener process ($Q=I$), with $\overline{\alpha}=\frac14$. Owing to the inequalities~\eqref{eq:sob} and~\eqref{eq:regul_alpha}, they are also satisfied if $\overline{\alpha}>\frac14$. This includes the case of a trace-class noise ($\sum_{j\in\N}q_j<\infty$), for which $\overline{\alpha}=\frac12$.

In the sequel, to simplify notation let $Z_n=Z_n^{\Delta t}$ and $\tilde{Z}(t)=\tilde{Z}^{\Delta t}(t)$.

\section{Preliminary results}\label{sec:recall}

This section is devoted to state and prove some properties of the solution of the stochastic evolution equation
\begin{equation}\label{eq:SPDE}
dX(t)=AX(t)dt+F(X(t))dt+dW^Q(t)~,\quad X(0)=x_0.
\end{equation}
Define $Y(t)=X(t)-Z(t)$ for all $t\ge 0$. Then~\eqref{eq:SPDE} is equivalent to the evolution equation with a random time-dependent nonlinearity
\[
\frac{dY(t)}{dt}=AY(t)+F(Y(t)+Z(t))~,\quad Y(0)=x_0.
\]
Let us first recall a well-posedness result. Let Assumptions~\ref{ass:poly} and~\ref{ass:momentZ} be satisfied. Then, for any initial condition $x_0\in L^q$, there exists a unique global mild solution~\eqref{eq:SPDE}, defined by
\[
X(t)=e^{tA}x_0+\int_{0}^{t}e^{(t-s)A}F(X(s))ds+\int_0^te^{(t-s)A}dW^Q(s)~,\quad t\ge 0.
\]
To put emphasis on the role of the initial condition $x_0$, the notation $X_{x_0}(t)=X(t)$ may be used.

The process $\bigl(X_{x_0}(t)\bigr)$ takes values in $L^\infty$. If Assumption~\ref{ass:condition} is satisfied, one has the following moment bounds in the $L^\infty$ norm, uniformly in time, if $x_0\in L^\infty$.
\begin{propo}\label{propo:momentX}
Let Assumptions~\ref{ass:poly},~\ref{ass:condition} and~\ref{ass:momentZ} be satisfied.

For all $m\in\N$, there exists $C_m\in(0,\infty)$ such that for all $x_0\in L^\infty$, one has
\[
\underset{t\ge 0}\sup~\bigl(\E[\|X_{x_0}(t)\|_{L^\infty}^m]\bigr)^{\frac1m}\le C_m(1+\|x_0\|_{L^\infty}^{q}).
\]
\end{propo}

\begin{proof}
Let $Y_{x_0}(t)=X^{x_0}(t)-Z(t)$. Then for all $t>0$, one has
\[
\frac{dY_{x_0}(t)}{dt}=AY_{x_0}(t)+F(Y_{x_0}(t)+Z(t)).
\]
By an energy estimate, using~\eqref{eq:condition} (see Assumption~\ref{ass:condition}) and Young's inequality one obtains
\[
\frac{1}{q}\frac{d\|Y_{x_0}(t)\|_{L^q}^{q}}{dt}\le -\gamma'\|Y_{x_0}(t)\|_{L^q}^q+C\|F(Z(t))\|_{L^q}^q,
\]
for some $\gamma'\in(0,\gamma)$ and $C\in(0,\infty)$.

Using Gronwall's lemma and the polynomial growth assumption for $F$, see Assumption~\ref{ass:poly}, for all $t\ge 0$ one has
\[
\|Y_{x_0}(t)\|_{L^q}^q\le e^{-q\gamma' t}\|x_0\|_{L^q}^q+C\int_{0}^{t}e^{-q\gamma'(t-s)}(1+\|Z(s)\|_{L^{q^2}}^{q^2})ds.
\]
Using Jensen's inequality and Assumption~\ref{ass:momentZ}, one obtains the inequality
\[
\underset{t\ge 0}\sup~\bigl(\E[\|Y_{x_0}(t)\|_{L^q}^m]\bigr)^{\frac1m}\le C_m(1+\|x_0\|_{L^q}^q),
\]
for some $C_m\in(0,\infty)$, for all $m\in\N$.

To obtain moment bounds in the $L^\infty$ norm, the argument uses the inequality~\eqref{eq:ineqLinftyL1}. First, the mild formulation yields the identity
\[
Y_{x_0}(t)=e^{tA}x_0+\int_0^t e^{(t-s)A}F(Y_{x_0}(s)+Z(s))ds,
\]
and applying~\eqref{eq:ineqLinftyL1} one has
\begin{align*}
\|Y_{x_0}(t)\|_{L^\infty}&\le \|x_0\|_{L^\infty}+C\int_0^t \frac{e^{-c(t-s)}}{(t-s)^\frac12}\|F(Y_{x_0}(s)+Z(s))\|_{L^1}ds\\
&\le \|x_0\|_{L^\infty}+C\int_0^t \frac{e^{-c(t-s)}}{(t-s)^\frac12}(1+\|Y_{x_0}(s)\|_{L^q}^q\|+Z(s))\|_{L^q}^q)ds.
\end{align*}
Applying Jensen's inequality and using the moment bounds in the $L^q$ norm proved above and Assumption~\ref{ass:momentZ} yields
\[
\underset{t\ge 0}\sup~\bigl(\E[\|Y_{x_0}(t)\|_{L^\infty}^m]\bigr)^{\frac1m}\le C_m(1+\|x_0\|_{L^\infty}^{q}),
\]
and combining again this result with Assumption~\ref{ass:momentZ} concludes the proof.
\end{proof}

\begin{rem}
The proof of Proposition~\ref{propo:momentX} is limited to the one-dimensional case, and it would be required to modify some assumptions in order to cover higher-dimensional situations. Precisely, if the interval $(0,1)$ is replaced by $(0,1)^d$ in the setting, then~\eqref{eq:ineqLinftyL1} holds with the exponent $-\frac12$ replaced by $-\frac{d}{2}$. As soon as $d\ge 2$, a non-integrable singularity appears and one can not retrieve a $L^\infty$ bound from a $L^q$ bound of the solution as in the last part of the proof of Proposition~\ref{propo:momentX}.

If $d=2$, it suffices to assume that the condition~\eqref{eq:condition} holds for some $p=\tilde{q}>q$, in order to apply a similar argument. If $d\ge 3$, it suffices to assume that the condition~\eqref{eq:condition} holds for some sufficiently large $p=\tilde{q}_d$ depending on dimension.

The analysis is presented in the case $d=1$ only in order to simplify the presentation.
\end{rem}

\begin{propo}\label{propo:ergoX}
Let Assumptions~\ref{ass:poly},~\ref{ass:condition} and~\ref{ass:momentZ} be satisfied. Then for all $x_0^1,x_0^2\in L^p$, with $p\in\{2,q\}$, one has for all $t\ge 0$,
\begin{equation}\label{eq:contractionX}
\|X_{x_0^2}(t)-X_{x_0^1}(t)\|_{L^p}\le e^{-\gamma t}\|x_0^2-x_0^1\|_{L^p}.
\end{equation}
Moreover, there exists a probability distribution $\mu_\star$, which satisfies $\int\|x\|_{L^\infty}^m d\mu_\star(x)<\infty$ for all $m\in\N$, such that for all $x_0\in L^p$ and all Lipschitz continuous functions $\varphi:L^p\to\R$, with $p\in\{2,q\}$, one has
\begin{equation}\label{eq:ergoX}
\big|\E[\varphi(X_{x_0}(t))]-\int\varphi d\mu_\star\big|\le C(\varphi)e^{-\gamma t}(1+\|x_0\|_{L^p}).
\end{equation}
Finally, $\mu_\star$ is the unique invariant probability distribution of~\eqref{eq:SPDE}.
\end{propo}

\begin{proof}
The construction is standard, so only a sketch of proof is provided. See for instance~\cite{DPZergo} for details.

Let $\delta X(t)=X_{x_0^2}(t)-X_{x_0^1}(t)$, then
\[
\frac{\delta X(t)}{dt}=A\delta X(t)+F(X_{x_0^2}(t))-F(X_{x_0^1}(t)),
\]
and using an energy estimate combined with~\eqref{eq:condition} yields
\[
\frac{1}{p}\frac{\|X_{x_0^2}(t)-X_{x_0^1}(t)\|_{L^p}^p}{dt}\le -\gamma\|X_{x_0^2}(t)-X_{x_0^1}(t)\|_{L^p}^p.
\]
Using Gronwall's lemma concludes the proof of~\eqref{eq:contractionX}.

Using the remote initial condition method, the contraction property~\eqref{eq:contractionX} allows to construct a stationary process $\bigl(X_\star(t)\bigr)_{t\ge 0}$, such that the distribution of $X_\star(t)$ is independent of $t$. Let $\mu_\star$ be this probability distribution. Using that construction, the property $\int \|x\|_{L^\infty}^md\mu_\star(x)<\infty$ is a consequence of the uniform moment bounds given by Proposition~\ref{propo:momentX}. Finally, the convergence to equilibrium property~\eqref{eq:ergoX} is a straightforward consequence of~\eqref{eq:contractionX} combined with the moment bound property above. The fact that the invariant probability distribution is unique is a direct consequence of~\eqref{eq:ergoX}.
\end{proof}

\section{Tamed exponential Euler scheme and main results}\label{sec:main}

In this work, we consider the explicit tamed exponential Euler scheme defined by
\begin{equation}\label{eq:scheme}
X_{n+1}=e^{\Delta tA}X_{n}+(-A)^{-1}(I-e^{\Delta tA})\frac{F(X_n)}{1+\Delta t\|F(X_n)\|_{L^2}}+e^{\Delta tA}\Delta W_n^Q,
\end{equation}
where as explained above $\Delta t\in(0,\Delta t_0]$ is the time-step size, $t_n=n\Delta t$ and $\Delta W_n^Q=W^Q(t_{n+1})-W^Q(t_n)$.

Define for all $t\ge 0$
\begin{equation}\label{eq:Xtilde}
\tilde{X}(t)=e^{tA}x_0+\int_{0}^{t}e^{(t-s)A}\frac{F(X_{\ell(s)})}{1+\Delta tM_{\ell(s)}}ds+\tilde{Z}(t)
\end{equation}
where $\tilde{Z}(t)=\tilde{Z}^{\Delta t}(t)$ is given by~\eqref{eq:Ztilde}, and $M_n=\|f(X_n)\|_{L^2}$. Note that $\tilde{X}(t_n)=X_n$ for all $n\ge 0$.

The first main result of this article states moment bounds in the $L^\infty$ norm for the auxiliary process $\bigl(\tilde{X}(t)\bigr)_{0\le t\le T}$, for all $T\in(0,\infty)$, with a dependence which is at most polynomial, uniformly with respect to the time-step size.
\begin{theo}\label{theo:moment}
Let Assumptions~\ref{ass:poly},~\ref{ass:condition} and~\ref{ass:momentZ} be satisfied. Let the process $\bigl(\tilde{X}(t)\bigr)_{t\ge 0}$ be defined by~\eqref{eq:Xtilde}.

For all $m\in\N$, there exists a polynomial function $\mathcal{P}_m:\R\to\R$ such that for all $T\in(0,\infty)$ and all $x_0\in L^\infty$, one has
\begin{equation}\label{eq:theo-moment_Xn}
\underset{\Delta t\in(0,\Delta t_0]}\sup~\underset{0\le n\le N}\sup~\bigl(\E[\|X_n\|_{L^\infty}^m]\bigr)^{\frac1m}\le \bigl(1+N\Delta t\bigr)\mathcal{P}_m(\|x_0\|_{L^\infty}).
\end{equation}
and
\begin{equation}\label{eq:theo-moment}
\underset{\Delta t\in(0,\Delta t_0]}\sup~\underset{0\le t\le T}\sup~\bigl(\E[\|\tilde{X}(t)\|_{L^\infty}^m]\bigr)^{\frac1m}\le (1+T^q)\mathcal{P}_m(\|x_0\|_{L^\infty}).
\end{equation}
\end{theo}
The proof of Theorem~\ref{theo:moment} is postponed to Section~\ref{sec:proofTh1}. It would be possible to express the degree of the polynomial function $\mathcal{P}_m$ in terms of $m$ and $q$, but since this does not matter below this dependence is not indicated.

Since the right-hand side of~\eqref{eq:theo-moment} depends on the time $T$, the moment bounds for the discrete time process qualitatively differ from the continuous time case, where uniform in time moment bounds hold, see Proposition~\ref{propo:momentX}. 

Since uniform in time moment bounds are not proven, then the discrete-time process may not admit an invariant probability distribution. Despite the loss of uniform in time moment bounds, it is in fact possible to approximate averages $\int \varphi d\mu_\star$ with respect to the invariant distribution $\mu_\star$, as explained by the second main result of this article.
\begin{theo}\label{theo:error}
There exists a polynomial function $\mathcal{P}:\R\to\R$, $Q\in\N$, and for every $\alpha\in(0,\overline{\alpha})$ and any function $\varphi:L^2\to\R$ of class $\mathcal{C}^2$ with bounded first and second order derivatives, there exists $C_\alpha(\varphi)\in(0,\infty)$, such that for all $x_0\in L^\infty$, all $\Delta t\in(0,\Delta t_0]$ and $N\in\N$, one has
\[
|\E[\varphi(X_N)]-\int\varphi d\mu_\star|\le C_\alpha(\varphi)\Delta t^{2\alpha}\bigl(1+(N\Delta t)^{Q}\bigr)\mathcal{P}(\|x_0\|_{L^\infty})+e^{-cN\Delta t}(1+\|x_0\|_{L^\infty}).
\]
\end{theo}
Recall that $\overline{\alpha}\in(0,\frac12]$ is given by~\eqref{eq:alphabar} (under Assumption~\ref{ass:alpha}), and that $\overline{\alpha}=\frac14$ if $Q=I$ ($q_j=1$ for all $j\ge 1$, space-time white noise) and $\overline{\alpha}=\frac12$ if ${\rm Tr}(Q)<\infty$ ($\sum_{j\in\N}q_j<\infty$, trace-class noise). The order of convergence for the temporal discretization error in Theorem~\ref{theo:error} thus coincides with the weak order of convergence for Euler type methods applied to the SPDE~\eqref{eq:SPDE}, either on finite time horizon (see~\cite{C-E:20} where $Q=I$ and $F$ is a cubic nonlinearity) or for the approximation of the invariant distribution for globally Lipschitz nonlinearities (see~\cite{C-E:14} where $Q=I$). The order of convergence $2\overline{\alpha}$ is optimal, it corresponds to the weak order for the discretization of the stochastic convolution, indeed one has
\[
\underset{n\in\N}\sup~\big|\E[\varphi(Z_n)]-\E[\varphi(Z(t_n)]\big|\le C_\alpha(\varphi)\Delta t^{2\alpha},
\]
where $\bigl(Z(t)\bigr)_{t\ge 0}$ and $\bigl(Z_n\bigr)_{n\in\N}$ are defined by~\eqref{eq:Z} and~\eqref{eq:Zn} respectively. The assumption that $\varphi$ is of class $\mathcal{C}^2$ cannot be relaxed, see~\cite{C-E:20_1} (in the case $Q=I$ and $F=0$).

Note that the proof gives the upper bound $Q\le 3q^2$, however the exact value of the integer $Q$ is not important for the analysis of the computational cost below. In fact, the same analysis of the cost would give the same conclusions even if $Q=0$ ({\it i.e.} if the moment bounds and error estimates would be uniform in time), as explained below.

Contrary to existing results concerning the numerical approximation of the invariant distribution for SPDEs, the weak error estimate from Theorem~\ref{theo:error} is not uniform with respect to time, instead the dependence is at most polynomial, with an upper bound for the degree given by $Q$. However, this at most polynomial growth does not lead to a reduction of the efficiency of the approach, indeed one has the following corollary.
\begin{cor}\label{cor:cost}
Let $x_0\in L^\infty$ and $\varphi:L^2\to\R$ of class $\mathcal{C}^2$ with bounded first and second order derivatives. For all $\alpha\in(0,\overline{\alpha})$, there exists $C_\alpha(x_0,\varphi)$ such that can choose $N$ and $\Delta t$ such that for all $\varepsilon\in(0,1)$, the error satisfies
\[
|\E[\varphi(X_N)]-\int\varphi d\mu_\star|\le \varepsilon,
\]
with a computational cost, defined as the number of time steps, satisfying
\[
\mathcal{C}(\varepsilon)=N\le C_\alpha(x_0,\varphi)\epsilon^{-\frac{1}{\alpha}}.
\]
\end{cor}
In the computational cost analysis, the spatial discretization of the process and the Monte-Carlo approximation of the expectation are not taken into account. In addition, it is assumed that the cost is of size $1$ at each step. Indeed, if fully-discrete schemes using the same spatial discretization and Monte-Carlo parameters are considered, it is sufficient to compare the temporal discretization error and the associated cost.
\begin{proof}[Proof of Corollary~\ref{cor}]
Let $\alpha\in(0,\overline{\alpha})$.

The parameters $N$ and $\Delta t$ are chosen such that
\[
N\Delta t=C|\log(\varepsilon)|
\]
and
\[
(N\Delta t)^{Q}\Delta t^{\frac{\alpha+\overline{\alpha}}{2}}=C\varepsilon,
\]
where $C$ is a constant (depending on $x_0$ and $\varphi$).

This leads to
\[
\Delta t=C\varepsilon^{\frac{2}{\alpha+\overline{\alpha}}}|\log(\varepsilon)|^{-\frac{2Q}{\alpha+\overline{\alpha}}},
\]
and finally
\[
N=C|\log(\varepsilon)|\Delta t^{-1}=C
\varepsilon^{-\frac{2}{\alpha+\overline{\alpha}}}||\log(\varepsilon)|^{1+\frac{2Q}{\alpha+\overline{\alpha}}}\le C\varepsilon^{-\frac{1}{\alpha}},
\]
since $\alpha<\overline{\alpha}$ and $|\log(\varepsilon)|^M={\rm o}(\varepsilon^{\frac1M})$ when $\varepsilon\to 0$, for all $M\in\N$.

This concludes the proof of Corollary~\ref{cor:cost}.
\end{proof}
Observe that Corollary~\ref{cor:cost} would be the same if one would have uniform in time estimates (as in the case of Lipschitz continuous coefficients, see~\cite{C-E:14}, or using an implicit scheme, see~\cite{Cui-Hong-Sun}) in Theorems~\ref{theo:moment} and~\ref{theo:error} instead of the at most polynomial dependence we prove in this article. However, the scheme studied in~\cite{C-E:14} cannot be employed for non globally Lipschitz nonlinearities, and the scheme from~\cite{Cui-Hong-Sun} is implicit: the scheme studied in this article is explicit thus is expected to be more efficient in practice. The exact rate of polynomial growth is not important in the argument above.

In addition, observe that obtaining a polynomial dependence with respect to time in the right-hand side of~\eqref{eq:theo-moment} is fundamental: had the dependence been exponential ($T$ replaced by $\exp(cT)$ in the right-hand side of~\eqref{eq:theo-moment}), a reduction of the effective order of convergence appearing in the analysis of the computational cost would have been observed, depending on the value of $c$.

This means that Theorem~\ref{theo:moment} is sufficiently good as far as one is concerned with the application to estimate averages $\int\varphi d\mu_\star$ with respect to the invariant distribution.

\begin{rem}
In~\cite{MilsteinTretyakov}, the authors propose to use the so-called rejecting exploding trajectories technique to approximate ergordic averages $\int \varphi d\mu_\star$, for SDEs with non-globally Lipschitz coefficients. This technique requires to introduce an auxiliary truncation parameter. However, even if in practice it is effective, this technique does not lead to a clean analysis of the cost as in Corollary~\ref{cor:cost}.
\end{rem}

\begin{rem}
Theorems~\ref{theo:moment} and~\ref{theo:error} hold if one can compute exactly in distribution $Z(t_n)$ for all $n\in\N$, with simpler proofs. This would mean replacing $Z_n$ defined by~\eqref{eq:Zn} and $\tilde{Z}(t)$ defined by~\eqref{eq:Ztilde}, by
\[
Z_{n+1}=e^{\Delta tA}Z_n+\int_{t_n}^{t_{n+1}}e^{(t_{n+1}-s)A}dW^Q(s)
\]
and $\tilde{Z}(t)=Z(t)$.
\end{rem}

\section{Proof of Theorem~\ref{theo:moment}}\label{sec:proofTh1}

This section is devoted to the proof of the first main result of this article, Theorem~\ref{theo:moment}.

The value of the time-step size $\Delta t\in(0,\Delta t_0]$ is fixed, and in the upper bounds obtained below the constants do not depend on $\Delta t$.

Recall that the process $\bigl(\tilde{X}(t)\bigr)_{t\ge 0}$ is defined by~\eqref{eq:Xtilde}, and that the auxiliary process $\bigl(\tilde{Z}(t)\bigr)_{t\ge 0}$ is defined by~\eqref{eq:Ztilde}. Owing to Assumption~\ref{ass:momentZtilde}, one has moment bounds in the $L^\infty$ norm for $\tilde{Z}(t)$, uniformly with respect to $t$.

Introduce an auxiliary parameter $R=\Delta t^{-\kappa}$, for a sufficiently small $\kappa$.

For every $n\ge 0$, let $\Omega_{R,t_{n}}=\{\underset{0\le \ell\le n}\sup~\|X_\ell\|_{L^\infty}\le R\}$, and to simplify notation let $\chi_n=\mathds{1}_{\Omega_{R,t_n}}$ denote the indicator function of the set $\Omega_{R,t_n}$. Let also $\chi_{-1}=1$.

Theorem~\ref{theo:moment} is then an straightforward  consequence of Lemma~\ref{lemma1} and~\ref{lemma2} stated below.
\begin{lemma}\label{lemma1}

For every $m\in\N$, there exists $C_m\in(0,\infty)$, such that for all $T\in(0,\infty)$ and $x_0\in L^\infty$, one has
\begin{equation}\label{eq:lemma1}
\underset{\Delta t\in(0,\Delta t_0]}\sup~\underset{0\le n\Delta t\le T}\sup~\E[\chi_{n-1}\|X_n\|_{L^\infty}^m]\le C_m(1+\|x_0\|_{L^\infty}^{mq^2}).
\end{equation}
\end{lemma}

\begin{lemma}\label{lemma2}

For every $m\in\N$, there exists a polynomial function $\mathcal{P}_m:\R\to\R$, such that for all $T\in(0,\infty)$ and $x_0\in L^\infty$, one has
\begin{equation}\label{eq:lemma2}
\underset{\Delta t\in(0,\Delta t_0]}\sup~\underset{0\le n\Delta t\le T}\sup~\E[(1-\chi_{n})\|X_n\|_{L^\infty}^m]\le (1+T^m)\mathcal{P}_m(\|x_0\|_{L^\infty}).
\end{equation}
\end{lemma}

Note that the moment bound~\eqref{eq:lemma1} is uniform with respect to time. The polynomial dependence in the right-hand side of~\eqref{eq:theo-moment} only comes from the use~\eqref{eq:lemma2}.

\begin{proof}[Proof of Lemma~\ref{lemma1}]

Introduce the following auxiliary processes:
\begin{align*}
\tilde{Z}_{x_0}(t)&=e^{\Delta t}Ax_0+\tilde{Z}(t)\\
\tilde{Y}(t)&=\tilde{X}(t)-\tilde{Z}_{x_0}(t)=\int_0^te^{(t-s)A}\frac{F(X_{\ell(s)})}{1+\Delta tM_{\ell(s)}}ds\\
R(t)&=\int_{0}^{t}e^{(t-s)A}F\bigl(\tilde{Y}(s)+\tilde{Z}_{x_0}(t_{\ell(s)})\bigr)ds\\
r(t)&=\tilde{Y}(t)-R(t).
\end{align*}

First, we claim that for all $p\in[2,\infty)$ and $m\in\N$, one has for all $0\le t< t_n$
\[
\E[\chi_{n-1}\|r(t)\|_{L^p}^m]\le C(1+\|x_0\|_{L^\infty}^{qm}),
\]

Indeed, $r(t)=r_1(t)+r_2(t)$ where
\begin{align*}
r_1(t)&=\int_0^t e^{(t-s)A}\frac{-\Delta tM_{\ell(s)}}{1+\Delta tM_{\ell(s)}}F({X}_{\ell(s)})ds\\
r_2(t)&=\int_0^t e^{(t-s)A}\bigl(F(\tilde{Y}(t_{\ell(s)})+\tilde{Z}_{x_0}(t_{\ell(s)}))-F(\tilde{Y}(s)+\tilde{Z}_{x_0}(t_\ell(s))\bigr)ds.
\end{align*}
On the one hand, using~\eqref{eq:expoLp} and the polynomial growth of $F$ (Assumption~\ref{ass:poly}), one has
\begin{align*}
\chi_{n-1}\|r_1(t)\|_{L^p}&\le \int_0^{t}e^{-c(t-s)}\Delta t\|F(X_{\ell(s)})\|_{L^p}^2ds\\
&\le C\chi_{n-1}\Delta t\int_{0}^{t}e^{-c(t-s)}(1+\|X_{\ell(s)}\|_{L^\infty}^{2q})ds\\
&\le C\Delta t(1+R^{2q})\\
&\le C,
\end{align*}
using $R=\Delta t^{-\kappa}$ with $\kappa<2q$.

On the other hand, using the polynomial growth of $F$ (Assumption~\ref{ass:poly}) and H\"older's inequality, one obtains
\begin{align*}
\big\|F(\tilde{Y}&(t_{\ell(s)})+\tilde{Z}(t_{\ell(s)}))-F(\tilde{Y}(s)+\tilde{Z}(t_\ell(s))\big\|_{L^p}\\
&\le C\|\tilde{Y}(t_{\ell(s)})-\tilde{Y}(s)\|_{L^{2p}}\bigl(1+\|\tilde{Y}(t_{\ell(s)})\|_{L^{2pq}}^{q}+\|\tilde{Y}(s)\|_{L^{2pq}}^q+\|\tilde{Z}_{x_0}(t_{\ell(s)})\|_{L^\infty}^q\bigr).
\end{align*}
Then, using~\eqref{eq:expoLp}, one has for all $t\in[0,t_n]$
\begin{align*}
\bigl(\E[\chi_{n-1}\|r_2(t)\|_{L^p}^m]\bigr)^{\frac1m}&\le C\int_{0}^{t}e^{-c(t-s)} \bigl(\E[\chi_{n-1}\big\|F(\tilde{Y}(t_{\ell(s)})+\tilde{Z}(t_{\ell(s)}))-F(\tilde{Y}(s)+\tilde{Z}_{x_0}(t_\ell(s))\big\|_{L^p}^{m}]\bigr)^{\frac1m}ds\\
&\le C\int_{0}^{t}e^{-c(t-s)}\mathcal{M}_1(s)\mathcal{M}_2(s)ds
\end{align*}
where
\begin{align*}
\mathcal{M}_1(s)&=\bigl(\E[\chi_{n-1}\|\tilde{Y}(t_{\ell(s)})-\tilde{Y}(s)\|_{L^{2p}}^{2m}]\bigr)^{\frac{1}{2m}}\\
\mathcal{M}_2(s)&=1+\bigl(\E[\chi_{n-1}\|\tilde{Y}(t_{\ell(s)})\|_{L^{2pq}}^{2qm}]\bigr)^{\frac{1}{2m}}+\bigl(\E[\chi_{n-1}\|\tilde{Y}(s)\|_{L^{2pq}}^{2qm}]\bigr)^{\frac{1}{2m}}+\bigl(\E[\|\tilde{Z}_{x_0}(s)\|_{L^\infty}^{2qm}]\bigr)^{\frac{1}{2m}}.
\end{align*}

To treat the first factor, let $s\in[t_k,t_{k+1})$, with $k\le n-1$, and let $\epsilon\in(0,1)$. Using~\eqref{eq:tempA}, one has
\begin{align*}
\|\tilde{Y}(s)-\tilde{Y}(t_{\ell(s)})\|_{L^{2p}}&\le \|(e^{(t-t_k)A}-I)\tilde{Y}(t_k)\|_{L^{2p}}+\|\int_{t_k}^{s}e^{(s-r)A}\frac{F(X_k)}{1+\Delta tM_k}dr\|_{L^{2p}}\\
&\le C_\epsilon \Delta t^\alpha\|(-A)^\alpha \tilde{Y}(t_k)\|_{L^{2p}}+C\Delta t\|F(X_k)\|_{L^{\infty}},
\end{align*}
where $\alpha\in(0,1)$.

Using~\eqref{eq:expoAalpha}, one obtains
\begin{align*}
\chi_{n-1}\|(-A)^\alpha \tilde{Y}(t_k)\|_{L^{2p}}\le C_\alpha\chi_{n-1}\int_0^t e^{-c(t-s)}(t-s)^{-\alpha}\|F(X_{\ell(s)})\|_{L^{2p}}ds\\ C(1+R^q),
\end{align*}
which gives
\[
\mathcal{M}_1(s)\le C\Delta t^\alpha (1+R^q).
\]

To treat the second factor, note that $\bigl(\E[\|\tilde{Z}_{x_0}(s)\|_{L^\infty}^{2qm}]\bigr)^{\frac{1}{2m}}\le C+\|x_0\|_{L^\infty}^q$, using the moment bound assumption for $\tilde{Z}$ (Assumption~\ref{ass:momentZtilde}). In addition, using~\eqref{eq:expoLp}, one has for $s<t_n$
\begin{align*}
\chi_{n-1}\|\tilde{Y}(s)\|_{L^p}&\le \chi_{n-1}\int_0^s e^{-c(s-r)}\|F(X_{\ell(r)})\|_{L^\infty}ds\\
&\le C(1+R^q),
\end{align*}
thus, one has for all $0\le s<t_n$
\[
\mathcal{M}_2(s)\le C(1+R^{q^2})+C\|x_0\|_{L^\infty}^q.
\]

Choosing $R=\Delta t^{-\kappa}$ with sufficiently small $\kappa$ then gives the upper bound $\mathcal{M}_1(s)\mathcal{M}_2(s)\le C$, for all $s\in[0,t_n)$, for some $C\in(0,\infty)$, and
\[
\bigl(\E[\chi_{n-1}\|r_1(t)\|_{L^p}^m]\bigr)^{\frac1m}+\bigl(\E[\chi_{n-1}\|r_2(t)\|_{L^p}^m]\bigr)^{\frac1m}\le C(1+\|x_0\|_{L^\infty}^q),
\]
for all $t\in[0,t_n)$. As a consequence, this concludes the proof of the claim.

Second, observe that
\[
\frac{dR(t)}{dt}=AR(t)+F\bigl(R(t)+r(t)+\tilde{Z}_{x_0}(t_{\ell(t)})\bigr).
\]
Using condition~\eqref{eq:condition} (Assumption~\ref{ass:condition}) and Young's inequality, one obtains
\[
\frac{1}{q}\frac{d\|R(t)\|_{L^q}^q}{dt}\le -\frac{\gamma}{2}\|R(t)\|_{L^q}^q+C\|F(r(t)+\tilde{Z}_{x_0}(t_{\ell(t)}))\|_{L^q}^q,
\]
and as a consequence one has for all $t\in[0,t_n]$,
\[
\chi_{n-1}\|R(t)\|_{L^q}^q\le C\int_0^t e^{-q\frac{\gamma}{2}(t-s)}\bigl(1+\chi_{n-1}\|r(s)\|_{L^{q^2}}^{q^2}+\|\tilde{Z}_{x_0}(t_{\ell(s)})\|_{L^\infty}^{q^2}\bigr)ds.
\]
Using the moment bound proved above for $r(s)$ and Assumption~\ref{ass:momentZtilde}, one obtains for all $t\in[0,t_n]$
\[
\bigl(\E[\chi_{n-1}\|R(t)\|_{L^q}^{m}]\bigr)^{\frac1m}\le C(1+\|x_0\|_{L^\infty}^q).
\]
Finally, since $\tilde{X}(t)=\tilde{Y}(t)+\tilde{Z}_{x_0}(t)=r(t)+R(t)+\tilde{Z}_{x_0}(t)$, one obtains for all $0\le t\le t_n$
\[
\bigl(\E[\chi_{n-1}\|\tilde{X}(t)\|_{L^q}^{m}]\bigr)^{\frac1m}\le C(1+\|x_0\|_{L^\infty}^{q^2}).
\]

It remains to prove the moment bound for the $L^\infty$ norm instead of the $L^q$ norm. This is obtained as follows: using~\eqref{eq:ineqLinftyL1}, for all $t\in[0,t_n]$, one has
\begin{align*}
\bigl(\E[\chi_{n-1}\|\tilde{X}(t)\|_{L^\infty}^m]\bigr)^{\frac1m}&\le \|e^{tA}x_0\|_{L^\infty}\\
&~+C\int_0^{t}e^{-c(t-s)}(t-s)^{-\frac12}\bigl(\E[\chi_{n-1}\|F(X_{\ell(s)})\|_{L^1}^m]\bigr)^{\frac1m}ds+\bigl(\E[\|\tilde{Z}(t)\|_{L^\infty}^m]\bigr)^{\frac1m}\\
&\le C+\|x_0\|_{L^\infty}+C\int_0^t e^{-c(t-s)}(t-s)^{-\frac12}\bigl(\E[\chi_{n-1}\|X_{\ell(s)}\|_{L^q}^{qm}])^{\frac1m}ds\\
&\le C(1+\|x_0\|_{L^\infty}^q).
\end{align*}
This concludes the proof of Lemma~\ref{lemma1}.
\end{proof}

\begin{proof}[Proof of Lemma~\ref{lemma2}]

Recall that $\chi_n=\mathds{1}_{\Omega_{R,t_{n}}}$, with $\Omega_{R,t_n}=\{\underset{0\le \ell\le n}\sup~\|X_\ell\|_{L^\infty}\le R\}$ and $\chi_{-1}=1$. As a consequence, one has
\begin{align*}
1-\chi_n&=\mathds{1}_{\Omega_{R,t_n}^c}=\mathds{1}_{\Omega_{R,t_{n-1}}^c}+\mathds{1}_{\Omega_{R,t_{n-1}}}\mathds{1}_{\|X_n\|_{L^\infty}>R}\\
&=1-\chi_{n-1}+\chi_{n-1}\mathds{1}_{\|X_n\|_{L^\infty}>R}.
\end{align*}
One thus obtains the equality
\[
1-\chi_n=\sum_{\ell=0}^{n}\chi_{\ell-1}\mathds{1}_{\|X_\ell\|_{L^\infty}>R}.
\]
Let $p\in\N$. Using Minkowksi, Cauchy-Schwarz and Markov inequalities, one obtains
\begin{align*}
\bigl(\E[(1-\chi_n)\|X_n\|_{L^\infty}^{m}]\bigr)^{\frac1m}&\le \sum_{\ell=0}^{n}\bigl(\E[\chi_{\ell-1}\mathds{1}_{\|X_\ell\|_{L^\infty}>R}\|X_n\|_{L^\infty}^m]\bigr)^{\frac1m}\\
&\le \sum_{\ell=0}^{n}\bigl(\E[\|X_n\|_{L^\infty}^{2m}]\bigr)^{\frac{1}{2m}}\bigl(\E[\chi_{\ell-1}\frac{\|X_\ell\|_{L^\infty}^\theta}{R^\theta}]\bigr)^{\frac{1}{2m}},
\end{align*}
where $\theta\in\N$ is chosen below.

On the one hand, by construction of the tamed Euler scheme and using~\eqref{eq:ineqLinftyL2} and Assumption~\ref{ass:momentZ}, one has
\begin{align*}
\|\tilde{X}(t)\|_{L^\infty}&\le \|x_0\|_{L^\infty}+C\int_0^t e^{-c(t-s)}(t-s)^{-\frac14}\frac{\|f(X_{\ell(s)})\|_{L^2}}{1+\Delta t\|f(X_{\ell(s)})\|_{L^2}}ds+\|Z(t)\|_{L^\infty}\\
&\le \|x_0\|_{L^\infty}+\frac{C}{\Delta t}+\|Z(t)\|_{L^\infty}
\end{align*}
thus 
\[
\bigl(\E[\|X_n\|^{2m}]\bigr)^{\frac{1}{2m}}\le C(1+\|x_0\|_{L^\infty}+\frac{1}{\Delta t}).
\]
On the other hand, applying Lemma~\ref{lemma1} yields for all $\ell\ge 0$
\[
\E[\chi_{\ell-1}\|X_\ell\|_{L^\infty}^\theta]\le C(1+\|x_0\|_{L^\infty}^{\theta q^2}).
\]
Gathering the estimates yields
\[
\bigl(\E[(1-\chi_n)\|X_n\|^{m}]\bigr)^{\frac1m}\le C\frac{T}{\Delta t}(1+\|x_0\|_{L^\infty}+\frac{1}{\Delta t})(1+\|x_0\|_{L^\infty}^{\theta q^2})R^{-\frac{\theta}{2m}}.
\]
Since $R=\Delta t^{-\kappa}$, it suffices to choose $\frac{\theta\kappa}{2m}>2$ in order to obtain~\eqref{eq:lemma2}.

This concludes the proof of Lemma~\ref{lemma2}.

\end{proof}

We are now in position to provide the proof of Theorem~\ref{theo:moment}.
\begin{proof}[Proof of Theorem~\ref{theo:moment}]
Since $X_n=\chi_nX_n+(1-\chi_n)X_n$, combining Minkowskii's inequality with~\eqref{eq:lemma1} and~\eqref{eq:lemma2} gives
\[
\sup~\underset{0\le n\le N}\sup~\bigl(\E[\|X_n\|_{L^\infty}^m]\bigr)^{\frac1m}\le C_m\bigl(1+N\Delta t\bigr)\mathcal{P}_m(\|x_0\|_{L^\infty}).
\]
This concludes the proof of~\eqref{eq:theo-moment_Xn}. It remains to prove~\eqref{eq:theo-moment}.

Using the notation from the proof of Lemma~\ref{lemma1}, one has $\tilde{X}(t)=\tilde{Y}(t)+\tilde{Z}_{x_0}(t)$, with
\[
\tilde{Y}(t)=\int_0^te^{(t-s)A}\frac{F(X_{\ell(s)})}{1+\Delta tM_{\ell(s)}}ds.
\]
Using~\eqref{eq:ineqLinftyL2}, one has, for all $t\in[0,T]$
\begin{align*}
\|\tilde{Y}(t)\|_{L^\infty}&\le C\int_0^t e^{-c(t-s)}(t-s)^{-\frac12}\|F(X_{\ell(s)}\|_{L^2}ds\\
&\le C\int_0^t e^{-c(t-s)}(t-s)^{-\frac12}(1+\|X_{\ell(s)}\|_{L^\infty}^q)ds.
\end{align*}
Using the moment bound for $X_n$ above, with $n\Delta t\le N\Delta t\le T$, one obtains
\[
\underset{0\le t\le T}\sup~\bigl(\E[\|\tilde{Y}(t)\|_{L^\infty}^m]\bigr)^{\frac1m}\le (1+T^q)\mathcal{P}_m(\|x_0\|_{L^\infty}),
\]
and combining this with~\eqref{eq:LinftyLinfty} with Assumption~\ref{ass:momentZtilde} concludes the proof of~\eqref{eq:theo-moment} and of Theorem~\ref{theo:moment}.
\end{proof}

\section{Proof of Theorem~\ref{theo:error}}\label{sec:proofTh2}

This section is devoted to the proof of the second main result of this article. The approach is based on the analysis of the weak error of the numerical scheme, where the dependence with respect to the final time $T=N\Delta t$ is carefully mentioned.

The weak error analysis uses the Kolmogorov equation approach. Some important regularity properties are given in Section~\ref{sec:Kolmo} below. A few auxiliary results concerning spatial and temporal regularity of $\tilde{Z}(t)$ and $\tilde{Y}[t)$ are stated and proved in Section~\ref{sec:aux}. Finally, weak error estimates are proved in Section~\ref{sec:weak}.

All the computations and statements hold rigorously using suitable approximations: the nonlinearity may be replaced by a globally Lipschitz continuous approximation and the noise may be truncated.  The objective is to prove that bounds hold independently of the approximation parameters. This is a standard approach in the analysis of Kolmogorov equations and weak error in infinite dimension: for instance one may consider an approximate stochastic evolution equation of the type
\[
dX^{\delta,\tau,J}(t)=AX^{\delta,\tau,J}dt+e^{\delta A}F_\tau(X^{\delta,\tau,J}(t))dt+e^{\delta A}\tilde{P}_JdW^Q(t),
\]
where $\delta>0$, $F_\tau(x)=\frac{1}{\tau}(\Phi_\tau(x)-x)$ where $\tau>0$ and $\bigl(\Phi_t\bigr)_{t\ge 0}$ the flow associated with the ordinary differential equation $\dot{z}=F(z)$, and $\tilde{P}_J=\sum_{j=1}^{J}\langle \cdot,\tilde{e}_j\rangle \tilde{e}_j$, $J\in\N$, is an orthogonal projection with finite rank. The noise is then finite dimensional, $F_\tau$ is Lipschitz continuous, and $e^{\delta A}$ is regularizing, so that all the computations make sense. In the estimates, the parameters $\delta,\tau,J$ do not appear and it suffices to pass to the limit $\delta\to 0,\tau\to 0,J\to\infty$ to get results for the model of interest and its numerical approximation. In order to simplify the notation, the approximation parameters are omitted in the sequel.

\subsection{Regularity results for the Kolmogorov equation}\label{sec:Kolmo}

Let $\varphi:L^2\to\R$ be a function of class $\mathcal{C}^2$. Define
\[
u(t,x)=\E[\varphi(X^x(t))],
\]
for all $t\ge 0$ and $x\in L^q$, where $\bigl(X^x(t)\bigr)_{t\ge 0}$ is the solution with initial condition $X^x(0)=x$.

Several properties of $u$ are employed in the weak error analysis. First, $u$ is solution of the Kolmogorov equation
\begin{equation}\label{eq:Kolmo}
\partial_t u(t,x)=\mathcal{L}u(t,x)=Du(t,x).\bigl(Ax+F(x)\bigr)+\frac12 \sum_{j\in\N}q_jD^2u(t,x).(\tilde{e}_j,\tilde{e}_j),
\end{equation}
with initial condition $u(0,\cdot)=\varphi$, where $Du(t,x)$ and $D^2u(t,x)$ are the first and second order derivatives of $u(t,x)$ with respect to the variable $x$.

Second, one has for all $T\ge 0$ and all $x\in L^q$,
\begin{equation}\label{eq:cvexpo}
\big|u(T,x)-\int \varphi d\mu_\star\big|\le C(\varphi)e^{-\gamma T}(1+\|x\|_{L^q}).
\end{equation}

Finally, some regularity properties for the spatial derivatives are instrumental in the weak error analysis below.

\begin{propo}\label{propo:Kolmogorov}
There exists $c\in(0,\infty)$, such that the following holds. First, for all $\alpha\in[0,1)$, there exists $C_\alpha\in(0,\infty)$ such that for all $t>0$, $x\in L^\infty$ and $h\in L^2$, one has
\[
|Du(t,x).h|\le C_\alpha(1+\|x\|_{L^\infty}^{q\mathds{1}_{\alpha>0}}) e^{-ct}\min(t,1)^{-\alpha}\|(-A)^{-\alpha}h\|_{L^2}.
\]
Moreover, for all $\beta_1,\beta_2\in[0,1)$, such that $\beta_1+\beta_2<1$, there exists $C_{\beta_1,\beta_2}\in(0,\infty)$ such that for all $t>0$, $x\in L^\infty$ and $h_1,h_2\in L^2$, one has
\[
|D^2u(t,x).(h_1,h_2)|\le C_{\beta_1,\beta_2}(1+\|x\|_{L^\infty}^{Q(\alpha_1,\alpha_2)})e^{-ct}\min(t,1)^{-\beta_1-\beta_2}\|(-A)^{-\beta_1}h_1\|_{L^2}\|(-A)^{-\alpha_2}h_2\|_{L^2}.
\]
with $Q(\alpha_1,\alpha_2)=q+q\mathds{1}_{\alpha_1>0}+\mathds{1}_{\alpha_2>0}$.
\end{propo}

Proposition~\ref{propo:Kolmogorov} is a variant of existing results: see~\cite{C-E:14} and~\cite{C-E:17} for the Lipschitz case. See also~\cite{C-E:20} and~\cite{CuiHong:19} for the treatment of polynomial nonlinearities, for estimates with $t\le T$.

\begin{proof}[Proof of Proposition~\ref{propo:Kolmogorov}]
The first and second order derivatives of $u(t,\cdot)$ are expressed as
\begin{align*}
Du(t,x).h&=\E[D\varphi(X^x(t)).\eta^h(t)]\\
D^2u(t,x).(h_1,h_2)&=\E[D\varphi(X^x(t)).\zeta^{h_1,h_2}(t)]+\E[D^2\varphi(X^x(t)).(\eta^{h_1}(t),\eta^{h_2}(t))],
\end{align*}
where the processes $\bigl(\eta^h(t)\bigr)_{t\ge 0}$ and $\bigl(\zeta^{h_1,h_2}(t)\bigr)_{t\ge 0}$ are solutions of
\begin{align*}
\frac{d}{dt}\eta^h(t)&=\bigl(A+F'(X^x(t))\bigr)\eta^h(t)\\
\frac{d}{dt}\zeta^{h_1,h_2}(t)&=\bigl(A+F'(X^x(t))\bigr)\zeta^{h_1,h_2}(t)+F''(X^x(t))\eta^{h_1}(t)\eta^{h_2}(t),
\end{align*}
with initial conditions $\eta^h(0)=h$ and $\zeta^{h_1,h_2}(0)=0$.

In the computations below, the values of $C$ and $c$ may change from line to line.

Introduce the family of operators $\bigl(\Theta(t,s)\bigr)_{t\ge s\ge 0}$, such that for all $h\in L^2$
\[
\frac{d}{dt}\Theta(t,s)h=\bigl(A+F'(X^x(t))\bigr)\Theta(t,s)h~,\quad t\ge s~,\quad \Theta(s,s)h=h.
\]
Using the condition~\eqref{eq:condition}, one has
\[
\frac{1}{2}\frac{d\|\Theta(t,s)h\|_{L^2}^2}{dt}\le -\gamma \|\Theta(t,s)\|_{L^2}^2,
\]
thus $\|\Theta(t,s)h\|_{L^2}\le e^{-\gamma (t-s)}\|h\|_{L^2}$ for all $t\ge s\ge 0$. This yields the result for $\alpha=0$.

To treat the case $\alpha>0$, introduce the auxiliary operators $\tilde{\Theta}(t,s)=\Theta(t,s)-e^{(t-s)A}$ for all $t\ge s$ (see~\cite{C-E:20}). Then one has
\[
\frac{d}{dt}\tilde{\Theta}(t,s)h=\bigl(A+F'(X^x(t))\bigr)\tilde{\Theta}(t,s)h+F'(X^x(t)e^{(t-s)A}h,
\]
with $\tilde{\Theta}(s,s)h=0$. Applying a Duhamel type formula yields
\[
\tilde{\Theta}(t,s)h=\int_s^t \Theta(t,r)\bigl(F'(X^x(r))e^{(r-s)A}h\bigr)dr.
\]
Using the result when $\alpha=0$, the polynomial growth of $F'$ (Assumption~\ref{ass:poly}) and the inequality~\eqref{eq:expoAalpha}, one obtains
\begin{align*}
\|\tilde{\Theta}(t,s)h\|_{L^2}&\le \int_{s}^{t}e^{-\gamma(t-r)}\|F'(X^x(r)))\|_{L^\infty}\|e^{(r-s)A}h\|_{L^2}dr\\
&\le C\int_{s}^{t}e^{-\gamma(t-r)}(1+\|X^x(r)\|_{L^\infty}^q)e^{-c(r-s)A}(r-s)^{-\alpha}dr \|(-A)^{-\alpha}h\|_{L^2},
\end{align*}
where $c\in(0,\gamma)>0$. Using moment estimates (see Proposition~\ref{propo:momentX}) and Jensen's inequality, one obtains
\[
\bigl(\E[\tilde{\Theta}(t,s)h\|_{L^2}^m]\bigr)^{\frac1m}\le C e^{-c'(t-s)}(1+\|x_0\|_{L^\infty}^q)\|(-A)^{-\alpha}h\|_{L^2}.
\]
Since $\Theta(t,s)=\tilde{\Theta}(t,s)+e^{(t-s)A}$, one finally obtains
\[
\bigl(\E[\Theta(t,s)h\|_{L^2}^m]\bigr)^{\frac1m}\le C e^{-c(t-s)}\min(t-s,1)^{-\alpha}(1+\|x_0\|_{L^\infty}^q)\|(-A)^{-\alpha}h\|_{L^2}.
\]
Since $\eta^h(t)=\Theta(t,0)h$, this gives the result when $\alpha\in(0,1)$ for the first-order derivative.

It remains to deal with the second-order derivative. On the one hand, since $\varphi$ is of class $\mathcal{C}^2$ with bounded second-order derivative, applying Cauchy-Schwarz inequality and the result above yields
\begin{align*}
\big|\E[D^2\varphi(X^x(t)).(\eta^{h_1}(t),\eta^{h_2}(t)]&\le C \bigl(\E[\|\eta^{h_1}(t)\|_{L^2}^2]\bigr)^{\frac12}\bigl(\E[\eta^{h_2}(t)\|_{L^2}^2]\bigr)^{\frac12}\\
&\le C \bigl(\E[\|\Theta(t,0)h_1\|_{L^2}^2]\bigr)^{\frac12}\bigl(\E[\Theta(t,0)h_2\|_{L^2}^2]\bigr)^{\frac12}\\
&\le Ce^{-2ct}t^{-\alpha_1-\alpha_2}(1+\|x_0\|_{L^\infty}^{q\mathds{1}_{\alpha_1>0}+\mathds{1}_{\alpha_2>0}})\|(-A)^{-\alpha_1}h_1\|_{L^2}\|(-A)^{-\alpha_2}h_2\|_{L^2}.
\end{align*}
On the other hand, a Duhamel type formula yields the equality
\[
\zeta^{h_1,h_2}(t)=\int_0^t \Theta(t,s)\bigl(F''(X^x(s))\eta^{h_1}(t)\eta^{h_2}(t)\bigr)ds.
\]
Let $\kappa\in(\frac14,1)$, such that inequality~\eqref{eq:sob} holds. By a duality argument, one has $\|(-A)^{-\kappa}\cdot\|_{L^2}\le C_\kappa\|\cdot\|_{L^1}$. Using the result above with a conditional expectation argument yields
\begin{align*}
\E[\|\zeta^{h_1,h_2}&(t)\|_{L^2}\le C\int_{0}^{t}\frac{e^{-c(t-s)}}{\min(t-s,1)^\kappa}(1+\|x_0\|_{L^\infty}^q)\bigl(\E[\|(-A)^{-\kappa}\bigl(F''(X^x(s))\eta^{h_1}(t)\eta^{h_2}(t)\bigr)\|_{L^2}^2]ds\\
&\le C\int_{0}^{t}\frac{e^{-c(t-s)}}{\min(t-s,1)^\kappa}(1+\|x_0\|_{L^\infty}^q)\bigl(\E\|\eta^{h_1}(t)\|_{L^2}^2]\E[\|\eta^{h_2}(t)\|_{L^2}^2]\bigr)^{\frac12}ds\\
&\le C(1+\|x_0\|_{L^q}^{Q(\alpha_1,\alpha_2)})\int_{0}^{t}\frac{e^{-c(t-s)}}{\min(t-s,1)^\kappa}\frac{e^{-2cs}}{\min(s,1)^{\alpha_1+\alpha_2}}ds\|(-A)^{-\alpha_1}h_1\|_{L^2}\|(-A)^{-\alpha_2}h_ 2\|_{L^2}\\
&\le C(1+\|x_0\|_{L^q}^{Q(\alpha_1,\alpha_2)})e^{-ct}\|(-A)^{-\alpha_1}h_1\|_{L^2}\|(-A)^{-\alpha_2}h_2\|_{L^2},
\end{align*}
using the condition $\alpha_1+\alpha_2$ to ensure integrability, where $Q(\alpha_1,\alpha_2)=q+q\mathds{1}_{\alpha_1>0}+\mathds{1}_{\alpha_2>0}$.

This gives
\[
|\E[D^2\varphi(X^x(t)).\zeta^{h_1,h_2}(t)]|\le C(1+\|x_0\|_{L^q}^{3q})e^{-ct}\|(-A)^{-\alpha_1}h_1\|_{L^2}\|(-A)^{-\alpha_2}h_2\|_{L^2}.
\]

Gathering the estimates then concludes the proof of Proposition~\ref{propo:Kolmogorov}.
\end{proof}

\subsection{Some useful regularity results}\label{sec:aux}

In this section, the objective is to state and prove some useful spatial and temporal regularity properties for $\tilde{Z}(t)$ defined by~\eqref{eq:Ztilde}, and for $\tilde{Y}(t)$ given by
\[
\tilde{Y}(t)=\int_{0}^{t}e^{(t-s)A}\frac{F(X_{\ell(s)})}{1+\Delta t\|F(X_{\ell(s)})\|_{L^2}}ds,
\]
which is such that $\tilde{X}(t)=e^{tA}x_0+\tilde{Y}(t)+\tilde{Z}(t)$.

In the sequel, let Assumption~\ref{ass:alpha} be satisfied, and let the parameter $\overline{\alpha}$ be defined by~\eqref{eq:alphabar}.

\begin{lemma}\label{lemma-aux1}
For every $\alpha\in[0,\overline{\alpha})$ and $m\in\N$, there exists $C_{\alpha,m}\in(0,\infty)$ such that
\[
\underset{n\ge 0}\sup~\E[\|(-A)^{\alpha}Z_n\|_{L^2}^m]\le C_{\alpha,m}
\]
and for all $n\ge 0$ and $t\in[t_n,t_{n+1}]$
\[
\bigl(\E[\|\tilde{Z}(t)-Z_n\|_{L^2}^m]\bigr)^{\frac1m}\le C_{\alpha,m}\Delta t^{\alpha}.
\]
\end{lemma}

\begin{proof}
It suffices to consider the case $m=2$ since $\tilde{Z}(t)$ is a Gaussian random variable with values in $L^2$ for all $t\ge 0$.

First, using It\^o's isometry formula and~\eqref{eq:expoAalpha}, one has
\begin{align*}
\E[\|(-A)^{\alpha}Z_n\|_{L^2}^2]&=\E[\|\sum_{\ell=0}^{n-1}(-A)^{\alpha}e^{(t_n-t_\ell)A}\Delta W_\ell^Q\|_{L^2}^2]\\
&=\Delta t\sum_{\ell=0}^{n-1}\sum_{j\in\N}q_j\|(-A)^{\alpha}e^{(t_n-t_\ell)A}\tilde{e}_j\|_{L^2}^2\\
&\le C_\epsilon \Delta t\sum_{\ell=0}^{n-1}\frac{e^{-c(t_n-t_\ell)}}{(t_n-t_{\ell})^{1-\epsilon}}\|(-A)^{\alpha+\frac{\epsilon}{2}-\frac12}\tilde{e}_j\|_{L^2}^2\\
&\le C_{\alpha,\epsilon}<\infty,
\end{align*}
if $\epsilon\in(0,\frac{\overline{\alpha}-\alpha}{2})$.

Second, using It\^o's isometry formula and the inequalities~\eqref{eq:tempA} and~\eqref{eq:expoAalpha}, one has for $t\in[t_n,t_{n+1}]$, 
\begin{align*}
\E[\|\tilde{Z}(t)-Z_n\|_{L^2}^2]&=\E[\|(e^{(t-t_n)A}-I)Z_n\|_{L^2}^2]+\E[\|e^{(t-t_n)A}(W^Q(t)-W^Q(t_n))\|_{L^2}^2]\\
&\le \Delta t^{2\alpha}\E[\|(-A)^\alpha Z_n\|_{L^2}^2]+\sum_{j\in\N}q_j(t-t_n)\|e^{(t-t_n)A}\tilde{e}_j\|_{L^2}^2\\
&\le C_\alpha\Delta t^{2\alpha}+(t-t_n)^{2\alpha}\sum_{j\in\N}q_j \|(t-t_n)^{\frac12-\alpha}(-A)^{\frac12-\alpha}e^{(t-t_n)A} (-A)^{\alpha-\frac12}\tilde{e}_j\|_{L^2}^2\\
&\le C_\alpha \Delta t^{2\alpha}.
\end{align*}
This concludes the proof of Lemma~\ref{lemma-aux1}.
\end{proof}

\begin{lemma}\label{lemma-aux2}
For every $\epsilon\in[0,1)$ and $m\in\N$, there exists  $C_{\epsilon,m}\in(0,\infty)$ and a polynomial function $\mathcal{P}_m:\R\to\R$ such that
\[
\underset{0\le n\Delta t\le N\Delta t}\sup~\bigl(\E[\|(-A)^{1-\epsilon}\tilde{Y}(t_n)\|_{L^2}^m]\bigr)^{\frac1m}\le C_{\alpha,m}(1+(N\Delta t)^q)\mathcal{P}_m(\|x_0\|_{L^\infty})
\]
and for all $n\ge 0$ and $t\in[t_n,t_{n+1}]$, with $n\le N$, one has
\[
\bigl(\E[\|\tilde{Y}(t)-\tilde{Y}(t_n)\|_{L^2}^m]\bigr)^{\frac1m}\le C_{\epsilon,m}\Delta t^{1-\epsilon}1+(N\Delta t)^q)\mathcal{P}_m(\|x_0\|_{L^\infty}).
\]
\end{lemma}

\begin{proof}
First, using~\eqref{eq:expoAalpha} and~\eqref{eq:theo-moment_Xn} for all $0\le t\le T$, one has
\begin{align*}
\bigl(\E\|(-A)^{1-\epsilon}\tilde{Y}(t)\|_{L^2}^{m}\bigr)^{\frac1m}&\le C\int_{0}^{t}\frac{e^{-c(t-s)}}{(t-s)^{1-\epsilon}}\bigl(\E[\|F(X_{\ell(s)})\|_{L^2}^m]\bigr)^{\frac1m}ds\\
&\le C\int_0^\infty\frac{e^{-cs}}{s^{1-\epsilon}}ds(1+T^q)\mathcal{P}_m(\|x_0\|_{L^\infty}).
\end{align*}
Second, for $0\le t_n\le t\le t_N$, one has
\begin{align*}
\bigl(\E[\|\tilde{Y}(t)-\tilde{Y}(t_n)\|_{L^2}^m]\bigr)^{\frac1m}&\le \bigl(\E[\|(e^{(t-t_n)A}-I)\tilde{Y}(t_n)\|_{L^2}^m]\bigr)^{\frac1m}\\
&~+\bigl(\E[\|\int_{t_n}^{t}e^{(t-s)A}\frac{F(X_n)}{1+\Delta t\|F(X_n)\|_{L^2}}ds\|_{L^2}^m]\bigr)^{\frac1m}ds\\
&\le C\Delta t^{1-\epsilon}\bigl(\E[\|(-A)^{1-\epsilon}\tilde{Y}(t_n)\|_{L^2}^m]\bigr)^{\frac1m}+\Delta t\bigl(\E[\|F(X_n)\|_{L^\infty}^m]\bigr)^{\frac1m}\\
&\le C\Delta t^{1-\epsilon}(1+T^q)\mathcal{P}_m(\|x_0\|_{L^\infty}).
\end{align*}
This concludes the proof of Lemma~\ref{lemma-aux2}.
\end{proof}

\subsection{Weak error analysis}\label{sec:weak}

We are now in position to study the weak error and prove Theorem~\ref{theo:error}.

The weak error is written as follows:
\begin{align*}
\E[\varphi(X_N)]-\E[\varphi(X(t_N))]&=\E[u(0,X_N)]-\E[u(t_N,X_0)]\\
&=\sum_{n=0}^{N-1}\bigl(\E[u(t_N-t_{n+1},X_{n+1})]-\E[u(t_N-t_n,X_{n})]\bigr)\\
&=\sum_{n=0}^{N-1}\bigl(\E[u(t_N-t_{n+1},\tilde{X}(t_{n+1}))]-\E[u(t_N-t_{n},\tilde{X}(t_n)]\bigr)\\
&=\sum_{n=0}^{N-1}\int_{t_n}^{t_{n+1}}\E\bigl[(-\partial_t+\mathcal{L}_n)u(t_N-t,\tilde{X}(t))\bigr]dt,
\end{align*}
using It\^o's formula, where for all $n\in\N$ the auxiliary operator $\mathcal{L}_n$ is defined by
\[
\mathcal{L}_n\phi=D\phi(x).(Ax+\frac{F(X_n)}{1+\Delta t\|F(X_n)\|_{L^2}})+\frac12\sum_{j\in\N}q_j D^2\phi(x).\bigl(e^{\Delta tA}\tilde{e}_j,e^{\Delta tA}\tilde{e}_j\bigr).
\]
Using the fact that $u$ is solution of the Kolmogorov equation~\eqref{eq:Kolmo}, one has
\[
\E[\varphi(X_N)]-\E[\varphi(X(t_N))]=\epsilon_N^1+\epsilon_N^2,
\]
where
\begin{align*}
\epsilon_N^1&=\int_{0}^{t_N}\E[Du(t_N-t,\tilde{X}(t)).\Bigl(\frac{F(X_{\ell(t)})}{1+\Delta t\|F(X_{\ell(t)})\|_{L^2}}-F(\tilde{X}(t))\Bigr)]dt\\
\epsilon_N^2&=\int_{0}^{t_N}\frac12\sum_{j\in\N}q_j\E[D^2u(t_N-t,\tilde{X}(t)).(e^{\Delta tA}\tilde{e}_j,e^{\Delta tA}\tilde{e}_j)-D^2u(t_N-t,\tilde{X}(t)).(\tilde{e}_j,\tilde{e}_j)]dt.
\end{align*}

Theorem~\ref{theo:error} is a straightforward consequence of Lemma~\ref{lem-error1} and Lemma~\ref{lem-error2} stated below.

\begin{lemma}\label{lem-error1}
There exists a polynomial function $\mathcal{P}:\R\to\R$, and for every $\alpha\in(0,\overline{\alpha})$ there exists $C_\alpha\in(0,\infty)$, such that for all $x_0\in L^\infty$, all $\Delta t\in(0,\Delta t_0]$ and $N\in\N$, one has
\[
|\epsilon_N^{2}|\le C_\alpha\Delta t^{2\alpha}\bigl(1+(N\Delta t)^{2q^2+2q}\bigr)\mathcal{P}(\|x_0\|_{L^\infty}).
\]
\end{lemma}

\begin{lemma}\label{lem-error2}
There exists a polynomial function $\mathcal{P}:\R\to\R$, and for every $\alpha\in(0,\overline{\alpha})$ there exists $C_\alpha\in(0,\infty)$, such that for all $x_0\in L^\infty$, all $\Delta t\in(0,\Delta t_0]$ and $N\in\N$, one has
\[
|\epsilon_N^{2}|\le C_\alpha\Delta t^{2\alpha}\bigl(1+(N\Delta t)^{3q^2}\bigr)\mathcal{P}(\|x_0\|_{L^\infty}).
\]
\end{lemma}

\begin{proof}[Proof of Theorem~\ref{theo:error}]
It suffices to write
\begin{align*}
|\E[\varphi(X_N)]-\int\varphi d\mu_\star|&\le |\E[\varphi(X_n)]-\E[\varphi(X(N\Delta t))]|+|\E[\varphi(X(N\Delta t))]-\int \varphi d\mu_\star|\\
&\le |\epsilon_N^1|+|\epsilon_N^2|+|u(N\Delta t,x_0)-\int\varphi d\mu_\star|,
\end{align*}
and to use Lemma~\ref{lem-error1}, Lemma~\ref{lem-error2} and the inequality~\ref{eq:cvexpo} to conclude the proof of Theorem~\ref{theo:error}.
\end{proof}

It remains to prove Lemma~\ref{lem-error1} and Lemma~\ref{lem-error2}. In the proofs, the notations $C$ and $\mathcal{P}$ is used for constants and polynomial functions respectively which may change from line to line. The dependence with respect to $T=t_N=N\Delta t$ is studied carefully.

\begin{proof}[Proof of Lemma~\ref{lem-error1}]
The error term $\epsilon_N^1$ is decomposed as follows:
\[
\epsilon_N^1=\epsilon_N^{1,1}+\epsilon_N^{1,2}+\epsilon_N^{1,3},
\]
where
\begin{align*}
\epsilon_N^{1,1}&=\int_{0}^{t_N}\E[-\Delta t\|F(X_{\ell(t)})\|_{L^2}Du(t_N-t,\tilde{X}(t)).F(X_{\ell(t)})]dt\\
\epsilon_N^{1,2}&=\int_{0}^{t_N}\E[\bigl(Du(t_N-t,\tilde{X}(t))-Du(t_N-t,X_{\ell(t)})\bigr).\bigl(F(X_{\ell(t)})-F(\tilde{X}(t))\bigr)]dt\\
\epsilon_N^{1,3}&=\int_{0}^{t_N}\E[Du(t_N-t,X_{\ell(t)}).\bigl(F(X_{\ell(t)})-F(\tilde{X}(t))\bigr)]dt.
\end{align*}

Using Proposition~\ref{propo:Kolmogorov} with $\alpha=0$ and Theorem~\ref{theo:moment}, for the first term one obtains
\begin{align*}
|\epsilon_{N}^{1,1}|&\le C\Delta t\int_{0}^{t_N}e^{-c(t_N-s)}\E[(1+\|\tilde{X}(t))\|_{L^\infty}^{q})(1+\|\tilde{X}(t_{\ell(t)})\|_{L^\infty}^{2q})]dt\\
&\le C\Delta t(1+t_N^{q^2+2q})\mathcal{P}(\|x_0\|_{L^\infty}).
\end{align*}

To study the second error term, using Proposition~\ref{propo:Kolmogorov} with $\alpha_1=\alpha_2=0$ and the polynomial growth of $F$, one obtains
\begin{align*}
|\epsilon_{N}^{1,2}|&\le C\int_{0}^{t_N}e^{-c(t_N-t)}\E\bigl[(1+\|\tilde{X}(t)\|_{L^\infty}^{2q}+\|\tilde{X}(t_{\ell(t)})\|_{L^\infty}^{2q})\|\tilde{X}(t)-\tilde{X}(t_{\ell(t)})\|_{L^2}^2\bigr]dt\\
&\le C\int_{0}^{t_N}e^{-c(t_N-t)}(1+\underset{0\le s\le t_N}\sup~\E[\|\tilde{X}(s)\|_{L^\infty}^{4q}])^{\frac12}\bigl(\E[\|\tilde{X}(t)-\tilde{X}(t_{\ell(t)})\|_{L^2}^4]\bigr)^{\frac12}dt.
\end{align*}
Recall that $\tilde{X}(t)=e^{tA}x_0+\tilde{Z}(t)+\tilde{Y}(t)$. Using the inequalities~\eqref{eq:tempA} and~\eqref{eq:expoAalpha}, one has
\[
\|e^{tA}x_0-e^{t_{\ell(t)}A}x_0\|_{L^2}\le C_\alpha\Delta t^{2\alpha}t_{\ell(t)}^{-2\alpha}\|x_0\|_{L^2}.
\]
for all $t\ge \Delta t$. Writing the integral for $t\in[0,t_N]$ as the sum as the integrals for $t\in[0,\Delta t]$ (which gives a contribution of size $\Delta t$) and $t\in[\Delta t,t_N]$, and using Lemma~\ref{lemma-aux1} and Lemma~\ref{lemma-aux2}, combined with the moment bounds from Theorem~\ref{theo:moment}, one obtains
\[
|\epsilon_N^{1,2}|\le C\Delta t^{2\alpha}(\mathcal{P}(\|x_0\|_{L^\infty})(1+\|(-A)^{2\alpha}x_0\|_{L^2}^2)(1+t_N^{2q^2+2q}).
\]
It remains to deal with the third error term $\epsilon_N^{1,3}$.

For every $n\ge 0$, let $\E_n=\E[\cdot|\mathcal{F}_{t_n}]$ denote the conditional expectation operator, where $\mathcal{F}_{t}=\sigma\bigl(W^Q(s),0\le s\le t\bigr)$. Set $\tilde{Y}_{x_0}(t)=e^{tA}x_0+\tilde{Y}(t)$, then one has $\tilde{X}(t)=\tilde{Y}_{x_0}(t)+\tilde{Z}(t)$. The error term $\epsilon_N^{1,3}$ is decomposed into two parts as follows:
\begin{align*}
\epsilon_N^{1,3}&=\int_{0}^{t_N}\E[Du(t_N-t,X_{\ell(t)}).\bigl(F(\tilde{Y}_{x_0}(t_{\ell(t)})+\tilde{Z}(t_{\ell(t)}))-F(\tilde{Y}_{x_0}(t)+\tilde{Z}(t))\bigr)]dt\\
&=\int_{0}^{t_N}\E[Du(t_N-t,X_{\ell(t)}).\bigl(F(\tilde{Y}_{x_0}(t_{\ell(t)})+\tilde{Z}(t))-F(\tilde{Y}_{x_0}(t)+\tilde{Z}(t))\bigr)]dt\\
&+\int_{0}^{t_N}\E[Du(t_N-t,X_{\ell(t)}).\bigl(F(\tilde{Y}_{x_0}(t_{\ell(t)})+\tilde{Z}(t_{\ell(t)}))-F(\tilde{Y}_{x_0}(t_{\ell(t)})+\tilde{Z}(t))\bigr)]dt\\
&=\epsilon_N^{1,3,1}+\epsilon_N^{1,3,2}.
\end{align*}
For the first term, using Proposition~\ref{propo:Kolmogorov} with $\alpha=0$, then moment bounds from Theorem~\ref{theo:moment}, and finally the inequalities~\ref{eq:tempA} and~\eqref{eq:expoAalpha} as above, and Lemma~\ref{lemma-aux2}, and one obtains
\begin{align*}
|\epsilon_N^{1,2,1}|&\le C\int_{0}^{t_N}e^{-c(t_N-t)}\E\bigl[\bigl(1+\|\tilde{Y}_{x_0}(t_{\ell(t)})\|_{L^\infty}^q+\|\tilde{Y}_{x_0}(t)\|_{L^\infty}^q+\|\tilde{Z}(t)\|_{L^\infty}^q\bigr)\|\tilde{Y}_{x_0}(t_{\ell(t)})-\tilde{Y}_{x_0}(t)\|_{L^2}\bigr]dt\\
&\le C\int_{0}^{t_N}e^{-c(t_N-t)}(1+T^{q^2})\mathcal{P}(\|x_0\|_{L^\infty})\bigl(\E[\|\tilde{Y}_{x_0}(t)-\tilde{Y}_{x_0}(t_{\ell(t)})\|_{L^2}^{2}]\bigr)^{\frac12}dt\\
&\le C(1+T^{q^2+q})\mathcal{P}(\|x_0\|_{L^\infty})\Delta t^{2\alpha}.
\end{align*}
The arguments for the second term $\epsilon_N^{1,3,2}$ are more involved. First, a conditional expectation argument gives
\begin{align*}
\epsilon_{N}^{1,2,2}&=\sum_{n=0}^{N-1}\int_{t_n}^{t_{n+1}}\E[Du(t_N-t,X_n).\bigl(F(\tilde{Y}_{x_0}(t_n)+Z_n)-F(\tilde{Y}_{x_0}(t_n)+\tilde{Z}(t))\bigr)]dt\\
&=\sum_{n=0}^{N-1}\int_{t_n}^{t_{n+1}}\E[Du(t_N-t,X_n).\bigl(F(\tilde{Y}_{x_0}(t_n)+Z_n)-\E_{n}[F(\tilde{Y}_{x_0}(t_n)+\tilde{Z}(t))]\bigr)]dt.
\end{align*}
A second-order Taylor expansion then yields
\begin{align*}
\E_{n}[F(\tilde{Y}_{x_0}(t_n)+\tilde{Z}(t))]-F(\tilde{Y}_{x_0}(t_n)+Z_n)&=F'(\tilde{Y}_{x_0}(t_n)).\E_n[\tilde{Z}(t)-Z_n]]+R_n\\
&=F'(\tilde{Y}_{x_0}(t_n)).\bigl((e^{\Delta tA}-I)Z_n\bigr)+R_n,
\end{align*}
with
\[
\|R_n\|_{L^1}\le C\bigl(1+\|\tilde{Y}_{x_0}(t_n)\|_{L^\infty}^q+\|Z_n\|_{L^\infty}^q+\|\tilde{Z}(t)\|_{L^\infty}^q\bigr)\|\tilde{Z}(t)-Z_n\|_{L^2}^2.
\]
Using Proposition~\ref{propo:Kolmogorov} and the inequality~\ref{eq:sob}, for $\kappa\in(\frac14,1)$, moment bounds from Theorem~\ref{theo:moment} and Assumption~\ref{ass:momentZtilde}, and Lemma~\ref{lemma-aux1}, one has
\begin{align*}
\big|\sum_{n=0}^{N-1}\int_{t_n}^{t_{n+1}}\E[Du(t_N-t,X_n).R_n]dt\big|&\le C\Delta t^{2\alpha}(1+T^{q^2})\mathcal{P}(\|x_0\|_{L^\infty}).
\end{align*}
To treat the last error term, combining the inequalities~\eqref{eq:ineqproduct1} and~\eqref{eq:ineqproduct2}, with polynomial growth of $F'$ (Assumption~\ref{ass:poly}) yields
\begin{align*}
\|(-A)^{-\alpha-2\epsilon}F'(\tilde{Y}_{x_0}(t_n))&.\bigl((e^{\Delta tA}-I)Z_n\bigr)\|_{L^1}\\
&\le C_{\alpha,\epsilon}(1+\|\tilde{Y}_{x_0}(t_n)\|_{L^\infty}^q)\|(-A)^{\alpha+\epsilon}\tilde{Y}_{x_0}(t_n)\|_{L^2}\|(-A)^{-\alpha}(e^{\Delta tA}-I)Z_n\|_{L^2}\\
&\le C_{\alpha,\epsilon}\Delta t^{2\alpha}(1+\|\tilde{Y}_{x_0}(t_n)\|_{L^\infty}^q)\|(-A)^{\alpha+\epsilon}\tilde{Y}_{x_0}(t_n)\|_{L^2}\|(-A)^{\alpha}Z_n\|_{L^2},
\end{align*}
where $\epsilon>0$, using the inequality~\eqref{eq:tempA} in the last step. Note that
\[
\|(-A)^{\alpha+\epsilon}\tilde{Y}_{x_0}(t_n)\|_{L^2}\le \|(-A)^{\alpha+\epsilon}\tilde{Y}(t_n)\|_{L^2}+C\|(-A)^{\alpha+\epsilon}e^{t_nA}x_0\|_{L^2}.
\]
As above, one uses the inequality~\eqref{eq:expoAalpha} to get $\|(-A)^{\alpha+\epsilon}e^{t_nA}x_0\|_{L^2}\le Ct_n^{-\alpha-\epsilon}\|x_0\|_{L^2}$ when $n\ge 1$, and a decomposition of the integral for $t\in[0,t_N]$ into integrals for $t\in[0,\Delta t]$ (which gives a contribution of size $\Delta t$) and $t\in[\Delta t,t_N]$. Using Proposition~\ref{propo:Kolmogorov} with $\alpha_1=\alpha+\kappa$ and $\alpha_2=0$, one then obtains
\begin{align*}
\big|&\sum_{n=0}^{N-1}\int_{t_n}^{t_{n+1}}\E[Du(t_N-t,X_n).F'(\tilde{Y}_{x_0}(t_n)).\bigl((e^{\Delta tA}-I)Z_n\bigr)]dt\big|\\
&\le C\Delta t^{2\alpha}\sum_{n=0}^{N-1}\int_{t_n}^{t_{n+1}}\frac{e^{-c(t_N-t)}}{(t_N-t)^{\alpha+\kappa+2\epsilon}}\E[(1+\|X_n\|_{L^\infty}^{2q})(1+\|\tilde{Y}_{x_0}(t_n)\|_{L^\infty}^q)\|(-A)^{\alpha+\epsilon}\tilde{Y}_{x_0}(t_n)\|_{L^2}\|(-A)^{\alpha}Z_n\|_{L^2}]dt\\
&\le C\Delta t^{2\alpha}(1+T^{q^2+2q})\mathcal{P}(\|x_0\|_{L^\infty}).
\end{align*}
owing to inequality~\eqref{eq:expoAalpha}, and to the moment bounds from Lemma~\ref{lemma-aux1} and~\ref{lemma-aux2} and Theorem~\ref{theo:moment}.

Finally, one has
\[
|\epsilon_N^{1,3,2}|\le C\Delta t^{2\alpha}(1+T^{q^2+2q})\mathcal{P}(\|x_0\|_{L^\infty}).
\]

Gathering the estimates for $\epsilon_N^{1,1}$, $\epsilon_N^{1,2}$ and $\epsilon_N^{1,3}$ then concludes the proof of Lemma~\ref{lem-error1}.
\end{proof}

\begin{proof}[Proof of Lemma~\ref{lem-error2}]
Using the symmetry of the bilinear operator $D^2u(T-t,\tilde{X}(t))$, one has
\[
\epsilon_N^2=\epsilon_N^{2,1}+\epsilon_N^{2,2},
\]
where
\begin{align*}
\epsilon_N^{2,1}&=\frac12\sum_{j\in\N}q_j\int_{0}^{t_N}\E[D^2u(T-t,\tilde{X}(t)).\bigl((e^{\Delta tA}-I)\tilde{e}_j,(e^{\Delta tA}-I)\tilde{e}_j\bigr)]dt\\
\epsilon_N^{2,2}&=\sum_{j\in\N}q_j\int_{0}^{t_N}\E[D^2u(T-t,\tilde{X}(t)).\bigl((e^{\Delta tA}-I)\tilde{e}_j,\tilde{e}_j\bigr)]dt.
\end{align*}

Let $\alpha\in(0,\overline{\alpha})$, and $\epsilon>0$ such that $\alpha+\epsilon<\overline{\alpha}$. Below, Proposition~\ref{propo:Kolmogorov} is used with $\alpha_1=\frac12+\alpha$ and $\alpha_2=\frac12-\alpha-\epsilon$. In addition, Theorem~\ref{theo:moment} is also used to control moments.

Using the inequalities~\eqref{eq:expoAalpha} and~\eqref{eq:tempA}, one obtains, with $T=N\Delta t$,
\begin{align*}
|\epsilon_N^{2,1}|&\le C\sum_{j\in\N}q_j\int_{0}^{t_N}\frac{e^{-c(t_N-t)}}{(t_N-t)^{1-\epsilon}}\|(-A)^{-\frac12-\alpha}(e^{\Delta tA}-I)\tilde{e}_j\|_{L^2}\|(-A)^{\alpha-\frac12+\epsilon}\tilde{e}_j\|_{L^2}(1+\E\|\tilde{X}(t)\|_{L^\infty}^{3q})dt\\
&\le C(1+T^{3q^2})\mathcal{P}(\|x_0\|_{L^\infty})\sum_{j\in\N}q_j\|(-A)^{-\frac12+\alpha}(-A)^{-2\alpha}(e^{\Delta tA}-I)\tilde{e}_j\|_{L^2}\|(-A)^{-\frac12+\alpha+\epsilon}\tilde{e}_j\|_{L^2}\\
&\le C(1+T^{3q^2})\mathcal{P}(\|x_0\|_{L^\infty})\Delta t^{2\alpha}\sum_{j\in\N}q_j\|(-A)^{\alpha-\frac12+\epsilon}\tilde{e}_j\|_{L^2}^2\\
&\le C(1+T^{3q^2})\mathcal{P}(\|x_0\|_{L^\infty})\Delta t^{2\alpha},
\end{align*}
since $\alpha+\epsilon<\overline{\alpha}$.

The second term is treated similar arguments: indeed one has
\begin{align*}
|\epsilon_N^{2,2}|&\le C\sum_{j\in\N}q_j\int_{0}^{t_N}\frac{e^{-c(t_N-t)}}{(t_N-t)^{1-\epsilon}}\|(-A)^{-\frac12-\alpha}(e^{\Delta tA}-I)\tilde{e}_j\|_{L^2}\|(-A)^{-\frac12+\alpha+\epsilon}\tilde{e}_j\|_{L^2}(1+\E\|\tilde{X}(t)\|_{L^\infty}^{3q})dt,
\end{align*}
and proceeding as above one obtains
\[
|\epsilon_N^{2,2}|\le C(1+T^{3q^2})\mathcal{P}(\|x_0\|_{L^\infty})\Delta t^{2\alpha}.
\]

This concludes the proof of Lemma~\ref{lem-error2}.
\end{proof}

\section{Acknoledgments}

This work is partially supported by the SIMALIN project ANR-19-CE40-0016 of the French National Research Agency.


\end{document}